\documentclass [a4paper,12pt]{article}
\usepackage{amsmath}
\usepackage{amsfonts}
\usepackage{indentfirst}
\usepackage{esint}
\usepackage{latexsym}
\usepackage{bbm}
\usepackage{amsfonts}
\usepackage{mathrsfs}
\usepackage{amssymb}
\usepackage{amsmath}
\usepackage{amsthm}
\usepackage{latexsym}
\usepackage{indentfirst}
\usepackage{enumitem}
\usepackage{bm}
\usepackage{esint}
\usepackage{cite}
\numberwithin{equation}{section}
\usepackage{color}
\allowdisplaybreaks
\usepackage{amsthm}
\usepackage{amssymb}
\usepackage{enumerate}
\usepackage{ulem}
\allowdisplaybreaks

\newtheorem{theorem}{Theorem}[section]
\newtheorem{lemma}[theorem]{Lemma}
\newtheorem{thm}[theorem]{Theorem}

\newtheorem{defn}[theorem]{Definition}
\newtheorem{rmk}[theorem]{Remark}
\newtheorem{pro}[theorem]{Proposition}
\makeatletter
\usepackage{amsmath}
\usepackage{amsfonts}
\usepackage{indentfirst}
\usepackage{esint}
\usepackage{latexsym}
\usepackage{color}

\usepackage{amsthm}
\usepackage{amssymb}
\usepackage{enumerate}
\usepackage{ulem}
\makeatletter

\newcommand{\Rmnum}[1]{\expandafter\@slowromancap\romannumeral #1@}
\makeatother
\begin{document}

\title{Quantitative Estimates on Reiterated Homogenization of Linear Elliptic operators Using Fourier Transform Methods}
\author{Yiping Zhang\footnote{Email:zhangyiping161@mails.ucas.ac.cn}\\Academy of Mathematics and Systems Science, CAS;\\
University of Chinese Academy of Sciences;\\
Beijing 100190, P.R. China.}
\date{}
\maketitle
\begin{abstract}
In this paper, we are interested in the reiterated homogenization of linear elliptic equations of the form $-\frac{\partial}{\partial x_{i}}\left(a_{i j}\left(\frac{x}{\varepsilon}, \frac{x}{\varepsilon^{2}}\right) \frac{\partial u_\varepsilon}{\partial x_{j}}\right)=f$ in $\Omega$ with Dirichlet boundary conditions. We obtain error estimates $O(\varepsilon)$ for a bounded  $C^{1,1}$ domain for this equation as well as the  interior
Lipschitz estimates at (very) large scale. Compared to the general homogenization problems, the difficulty in the reiterated homogenization is that we need to handle different scales of $x$. To overcome this difficulty, we firstly introduce the Fourier transform in the homogenization theory to separate these different scales. We also note that this method may be adapted to the following reiterated homogenization problem: $-\frac{\partial}{\partial x_{i}}\left(a_{i j}\left(\frac{x}{\varepsilon},\cdots, \frac{x}{\varepsilon^{N}}\right) \frac{\partial u_\varepsilon}{\partial x_{j}}\right)=f$ in $\Omega$ with Dirichlet boundary conditions. Moreover, our results may be extended to the related Neumann boundary problems without any real difficulty.
\end{abstract}

\section{Introduction and main results}
Before we state the introduction and the main results, we introduce the Einstein summation convention first.
Throughout this paper, we use the Einstein summation convention: An index occurring twice in a product is to be summed from 1 up to the space dimension, which means, for example,
\begin{equation*}
u_i v_i=\sum_{i=1}^{n}u_i v_i,
\end{equation*} if the space dimension is $n$.

The aim of the present paper is to study the error estimates and the interior Lipschitz estimates at large scale for linear elliptic equations, arising from
the reiterated homogenization problem. More precisely, let $\Omega\subset\mathbb{R}^n$ be a bounded domain with $n\geq 2$, and consider the following reiterated homogenization
problem in divergence form depending on a parameter $\varepsilon>0$,

\begin{equation}\label{1.1}
\left\{
\begin{aligned}
\mathcal{L}_{\varepsilon} u_{\varepsilon} \equiv-\frac{\partial}{\partial x_{i}}\left(a_{i j}\left(\frac{x}{\varepsilon}, \frac{x}{\varepsilon^{2}}\right) \frac{\partial u_\varepsilon}{\partial x_{j}}\right)&=f  \text { in } \Omega \\
u_{\varepsilon}&=g  \text { on } \partial \Omega.
\end{aligned}\right.
\end{equation}

Given three constants $\alpha$, $\beta$ and $M$ such that the function $a_{i j}(y,z) \in L^{\infty}\left(\mathbb{R}^{n} \times \mathbb{R}^{n}\right)$ satisfies the following conditions.\\
$\bullet $ The uniformly elliptic condition.
\begin{equation}\label{1.2}
\beta|\xi|^2\geq{a_{i j}(y, z) \xi_{i} \xi_{j} \geq \alpha |\xi|^2,\text{ for some } \beta\geq\alpha>0, \text { a.e. in } y, z.}
\end{equation}
$\bullet $ The smoothness condition. There exist a constant $M>0$, such that for any $y_1,y_2,z\in \mathbb{R}^n$, there holds
\begin{equation}\label{1.3}| a_{ij}(y_1,z)-a_{ij}(y_2,z)|\leq M|y_1-y_2|.\end{equation}
$\bullet $ The periodicity condition.
\begin{equation}\label{1.4}a_{i j}(y,z) \text { is } Y-Z \text { periodic.}\end{equation}

For simplicity, we may assume $Y=Z=(0,1)^n$. Denote $A=(a_{ij})$ is a $n\times n$ matrix.
The following homogenization results are well known ( See \cite[Chapter 1.8]{Louis1978Asymptotic}, for example).  Let $f\in H^{-1}(\Omega)$ and $u_\varepsilon$ be the weak solution to (\ref{1.1}). Then
$u_\varepsilon\rightharpoonup u_0$ weakly in $H^1(\Omega)$, and $A(x/\varepsilon,x/\varepsilon^2)\nabla u_\varepsilon\rightharpoonup\widehat{A}\nabla u_{0}$ weakly in $L^2(\Omega;\textrm{R}^n)$, where $u_0$ is the unique solution of
\begin{equation}\label{1.5}\left\{
\begin{aligned}
\mathcal{L}_{0} u_0 \equiv-\operatorname{div}\left(\widehat{A}\nabla u_0\right) &=f  \text { in } \Omega \\
u_{0}&=g  \text { on } \partial \Omega.
\end{aligned}\right.
\end{equation}
The operator $\widehat{A}=(\widehat{a}_{ij}):\Omega\times\mathbb{R}^n\mapsto\mathbb{R}^n$ is a constant matrix defined as
\begin{equation}\label{1.6}
\hat{a}_{i j}=\frac{1}{|Y||Z| }\iint_{Y \times Z}\left[a_{i j}-a_{i k} \frac{\partial \chi_{y}^{j}}{\partial z_{k}}-a_{i k} \frac{\partial \chi^{j}}{\partial y_{k}}+a_{i k} \frac{\partial \chi_{y}^{l}}{\partial z_{k}} \frac{\partial \chi^{j}}{\partial y_{l}}\right] d y d z,
\end{equation}
where $\chi_{y}^{k}(z)=\chi^{k}(y,z)$ is the unique solution of the cell-problem
\begin{equation}\label{1.7}\left\{
\begin{aligned}
-\frac{\partial}{\partial z_{i}}\left(a_{i j}\left(y, z\right) \frac{\partial}{\partial z_{j}}\left(\chi_{y}^{k}(z)-z_k\right)\right)=0  \text { in } Z \\
\chi_{y}^{k}(z)\  is\  Z-\text{periodic},\ \fint_Z\chi_{y}^{j}(z)dz=0,
\end{aligned}\right.
\end{equation}
for $k=1,2,\cdots,n$, and $\chi^k(z)$ is the unique solution of the cell problem
\begin{equation}\label{1.8}\left\{
\begin{aligned}
-\frac{\partial}{\partial y_{i}}\left[\left(\fint_Z\left(a_{i j}(y,z)-a_{i k}(y,z) \frac{\partial \chi_{y}^{j}(z)}{\partial z_{k}}\right)dz\right)\frac{\partial}{\partial y_j}(\chi^k(y)-y_k)\right]=0  \text { in } Y \\
\chi^{k}(y)\  is \ Y-\text{periodic},\ \fint_Y\chi^{k}(y)dy=0,
\end{aligned}\right.
\end{equation} for $k=1,2,\cdots,n$. Note the coefficient $\widehat{a}_{ij}$ and $b_{ij}(y)=\fint_Z\left(a_{i j}(y,z)-a_{i k}(y,z) \partial_{z_k} \chi_{y}^{j}(z)\right)dz$
satisfies the uniformly elliptic condition. \\
\begin{rmk}Due to $a_{ij}(y,z)$ is Y-Z periodic, then the solution $\chi^k_y(z)$ of the equation $(\ref{1.7})$ is also Y-Z periodic.
Actually, that $\chi^k_y(z)$ is periodic with respect to $y$ is an important observation, which is omitted in \cite[Chapter 1.8.5]{Louis1978Asymptotic}, and plays an essential role in Proposition 3.2.\end{rmk}
Throughout this paper, we use the following notation
$$H^m_{\text{per}(Y)}=:\left\{f\in H^m(Y) \text{ and }f\text{ is Y-periodic with }\fint_Yfdy=0\right\}. $$

The following theorem is the main result of the paper, which establish the $O(\varepsilon)$ convergence rates in $L^2(\Omega)$ for the Dirichlet problems.
\begin{thm} (convergence rates). Let $\Omega\subset \mathbb{R}^n$ be a bounded $C^{1,1}$ domain, and assume that $\mathcal{L}_\varepsilon$ satisfies the
conditions $(\ref{1.2}),\ (\ref{1.3})$ and $(\ref{1.4})$. If $f\in L^2(\Omega),$ and $g\in H^{3/2}(\partial \Omega)$, let $u_\varepsilon,u_0\in
H^1(\Omega)$ be the weak solutions of $(\ref{1.1})$ and $(\ref{1.5})$,  respectively, then there holds the following estimates \begin{equation*}
||u_\varepsilon-u_0||_{L^2(\Omega)}\leq C\varepsilon(||g||_{H^{3/2}(\partial \Omega)}+||f||_{L^2(\Omega)}),
\end{equation*} where $C$ depends on $\alpha,\beta,M,n$ and $\Omega$.\end{thm}

The convergence rate is one of the central issues in homogenization theory and has been studied extensively in the various setting. For elliptic equations and systems in divergence form with periodic coefficients, related results may be found in the recent work \cite{shen2018periodic,Kenig2010Homogenization,shen2017boundary,kenig2012convergence}.

As for the reiterated homogenization problems, very few results are known. In \cite[Chapter 1.8.5]{Louis1978Asymptotic}, the author has obtained
$$u_\varepsilon-u_0-\theta_\varepsilon\rightarrow 0\text{ strongly in }H^1_0(\Omega)\text{ as }\varepsilon\rightarrow 0,$$
if
\begin{equation}\label{1.9}
\begin{array}{l}{u_0 \in H^{2}(\Omega)} \\ {\chi^{j} \in W^{2, \infty}(Y)} \\ {\chi_{y}^{j}(z)=\chi^{j}(y, z) \in W^{1, \infty}(Y \times Z)},\end{array}
\end{equation}
where \begin{equation*}
\theta_{\varepsilon}=-\varepsilon \chi^{j}\left(\frac{x}{\varepsilon}\right) \frac{\partial u}{\partial x_{j}}(x) m_{\varepsilon}(x)-\varepsilon^{2} \chi_{y}^{j}\left(\frac{x}{\varepsilon^{2}}\right)\left[\frac{\partial u}{\partial x_{j}}-\left(\frac{\partial \chi^{k}}{\partial y_{j}}\right)\left(\frac{x}{\varepsilon}\right) \frac{\partial u}{\partial x_{k}}\right] m_{\varepsilon}
\end{equation*} for a suitable cut-off function $m_\varepsilon$. And in \cite{zhao2013convergence}, the convergence rate
$$||u_\varepsilon-u_0||_{L^\infty(\Omega)}\leq C\varepsilon$$
is obtained by a method based on the representation of elliptic equation solution by Green function, under the assumption
$$A(y,z)\in C^{1,\gamma}(Y\times Z) \text{ for some } \gamma>0,$$  which is a sufficient condition to ensure $(\ref{1.9})$. \\

Recently, the authors in \cite{Niuweisheng} have studied the reiterated homogenization problem of the form $-\operatorname{div}(A(x,x/\varepsilon_1,\cdots,x/\varepsilon_n)\nabla u_\varepsilon)=f$ with similar smooth assumptions on $A$ compared to this paper. To handle the different scales, the authors introduce the following $\varepsilon$-smoothing operator:
$$S_\varepsilon(g^\varepsilon)(x)=\int_{\mathbb{R^n}}g(z,x/\varepsilon)\rho_\varepsilon(x-z)dz,$$
where $\rho$ is a standard modifier.\\

Compared to \cite{Niuweisheng} (actually, this paper is a special case of \cite{Niuweisheng}), we firstly introduce the Fourier transform methods into homogenization problems to separate these different scales and obtain the $O(\varepsilon)$ error estimates. However, we must point out that the Fourier transform methods could't apply to the problem $-\operatorname{div}(A(x,x/\varepsilon_1,\cdots,x/\varepsilon_n)\nabla u_\varepsilon)=f$ unless $A$ is periodic with respect to the first variable.  For more details, see Proposition 3.2.\\

After obtaining the convergence rate Theorem 3.5, we may derive the following interior Lipschitz estimates at (very) large scale.
\begin{thm}(interior Lipschitz estimates at large scale) Assume that $A(y,z)$ satisfies the conditions $(\ref{1.2})$, $(\ref{1.3})$ and $(\ref{1.4})$.
 For some $x_0\in\mathbb{R}^n$, let $u_\varepsilon\in H^1(B(x_0,1))$ be a weak solution of $-\frac{\partial}{\partial x_{i}}\left(a_{i
 j}\left(\frac{x}{\varepsilon}, \frac{x}{\varepsilon^{2}}\right) \frac{\partial u_\varepsilon}{\partial x_{j}}\right)=f$ in $B(x_0,1)$, where $f\in
 L^p(B(x_0,1))$ for some $p>n$.Then for any $0<\varepsilon <1$, we have
\begin{equation}\label{1.10}
\left(\fint_{B\left(z_{0}, \varepsilon\right)}\left|\nabla u_{\varepsilon}\right|^{2}\right)^{1 / 2} \leq C\left\{\left(\fint_{B\left(z_{0}, 1\right)}\left|\nabla u_{\varepsilon}\right|^{2}\right)^{1 / 2}+\left(\fint_{B\left(z_{0}, 1\right)}|f|^{p}\right)^{1 / p}\right\},
\end{equation} where $C$ depends only on $\alpha,\beta,M$ and $n$.
\end{thm}
The Lipschitz estimate has been studied extensively in the various settings. For elliptic equations and systems in divergence form with periodic coefficients or almost periodic coefficients, related results may be found in the recent work \cite{shen2017boundary,shen2018periodic,armstrong2016lipschitz1}.\\

At this position, we give two remarks.
\begin{rmk}Similar to the proof of the interior Lipschitz estimates at large scale, we could obtain the boundary H\"{o}lder estimates at large scale under suitable boundary condition which we omit here (for more details, see \cite[Chapter 5.2]{shen2018periodic}).\end{rmk}
\begin{rmk}The scale of the interior Lipschitz estimates $(\ref{1.10})$ is too large for the variable $z$ to obtain the interior $W^{1,p}$ estimates for $u_\varepsilon$. Actually, we try to obtain the following estimates
\begin{equation*}
\left(\fint_{B\left(x_{0}, \varepsilon^2\right)}\left|\nabla u_{\varepsilon}\right|^{2}\right)^{1 / 2} \leq C\left\{\left(\fint_{B\left(x_{0}, 1\right)}\left|\nabla u_{\varepsilon}\right|^{2}\right)^{1 / 2}+\left(\fint_{B\left(x_{0}, 1\right)}|f|^{p}\right)^{1 / p}\right\},
\end{equation*} which is useful for obtaining the interior $W^{1,p}$ estimates, but we failed. For the reason, see Remark 3.3 and Remark 5.5.\end{rmk}

\section{Preliminaries}
\begin{lemma}
Let $\chi^k(y)$ and $\chi^k_y(z)$ be the weak solution of $(\ref{1.8})$ and $(\ref{1.7})$, respectively. Then there hold
\begin{equation}\label{2.1}\fint_Z\left|\nabla_y\chi^k_y(z)\right|^2dz+\fint_Z\left|\nabla_z\nabla_y\chi^k_y(z)\right|^2dz\leq C\end{equation} and
\begin{equation}\label{2.2}||\chi^k(y)||_{W^{2,p}(Y)}+||\chi^k_y(z)||_{W^{1,2}(Z)}\leq C\end{equation} for any $p\in(1,\infty)$ and $k=1,2,\cdots,n$, where $C$ depends on $\alpha,\beta,p,M$ and $n$.
\end{lemma}
\begin{proof}The proof is standard. Firstly, testing the equation $(\ref{1.7})$ with $\chi_y^k(z)$ gives that
\begin{equation}\label{2.3}||\chi^k_y(z)||_{W^{1,2}(Z)}\leq C.\end{equation} Then for any $y_1,y_2$, there holds
\begin{equation}\label{2.4}
-\partial_{z_i}\left((a_{ij}(y_1,z)-a_{ij}(y_2,z))\partial_{z_j}(\chi^k(y_1,z)-z_k)\right)=
\partial_{z_i}\left(a_{ij}(y_2,z)\partial_{z_j}(\chi^k(y_1,z)-\chi^k(y_2,z))\right),
\end{equation} due to $\chi^k(y_1,z)$ and $\chi^k(y_2,z)$ satisfying the equation $(\ref{1.8})$, respectively. Testing the equation $(\ref{2.4})$ with $\chi^k(y_1,z)-\chi^k(y_2,z)$, then it gives that
\begin{equation}\label{2.5}\begin{aligned}
\fint_Z|\nabla_z(\chi^k(y_1,z)-\chi^k(y_2,z))|^2dz&\leq C\fint_Z|a_{ij}(y_1,z)-a_{ij}(y_2,z)|^2(1+|\nabla_z(\chi^k(y_1,z)|^2)dz\\
&\leq C|y_1-y_2|^2
\end{aligned}\end{equation} due to $(\ref{1.3})$ and $(\ref{2.3})$, thus this together with Poinc\'{a}re inequality will give the state estimate $(\ref{2.1})$.

For the estimate $(\ref{2.2})$, we firstly note that $||\chi^k_y(z)||_{L^\infty(Z)}\leq C$ due to the De Giorgi-Nash-Moser theorem.
Then, we need only prove  $||\chi^k(y)||_{W^{2,p}(Y)}\leq C$. Actually, the equation $(\ref{1.8})$ reads as
\begin{equation}\label{2.6}b_{ij}(y)\partial^2_{y_iy_j}\chi^k(y)+\partial_{y_i}b_{ij}(y)\partial_{y_j}\chi^k(y)=\partial_{y_i}b_{ik}(y),\end{equation}
where $$b_{ij}(y)=\fint_Z\left(a_{i j}(y,z)-a_{i k}(y,z) \partial_{z_k} \chi_{y}^{j}(z)\right)dz.$$
Note that
$$\begin{aligned}|\nabla_yb_{ij}(y)|&\leq\fint_Z\left(|\nabla_ya_{i j}(y,z)|+|\nabla_ya_{i k}(y,z)|| \nabla_{z} \chi_{y}^{j}(z)|+a_{i k}(y,z) |\nabla_y\nabla_{z} \chi_{y}^{j}(z)|\right)dz\\
&\leq C,\end{aligned}$$
where we have used $(\ref{1.3})$, $(\ref{2.1})$ and$(\ref{2.3})$. Consequently, according to the $W^{2,p}$ estimates of elliptic equations in non-divergence form and $\chi^k(y)$ is Y-periodic, therefore we complete the proof of $(\ref{2.2})$.
\end{proof}

\begin{lemma}(reverse H\"{o}lder inequality). Let $\chi^k_y(z)$ be the weak solution to $(\ref{1.7})$, then there exists a constant $\tau>0$ which depends on $\alpha,\beta$ and $n$, such that for any $y$, there holds
\begin{equation}\label{2.7}
||\nabla_z \chi^k_y(z)||_{L^{2+\tau}(Z)}\leq C,
\end{equation} where $C$ depends on $\alpha,\beta,M$ and $n$.
\end{lemma}
\begin{proof}For any $z_0\in Z$, choose a cut-off function $\eta_r\in C^1_0(B(z_0,2r))$ satisfying $\eta_r=1$ in $B(z_0,r)$ and $\eta_r=0$ outside $B(z_0,3r/2)$ with $|\nabla \eta_r|\leq 4/r$. Testing the equation $(\ref{1.7})$ with $\eta^2_r(\chi^k_y(z)-c)$ gives that
\begin{equation}\label{2.8}\int_{B(z_0,r)}|\nabla_z \chi^k_y(z)|^2dz\leq \frac{C}{r^2}\int_{B(z_0,2r)}| \chi^k_y(z)-c|^2dz+Cr^n,\end{equation}
then, choose $c=\fint_{B(z_0,2r)}\chi^k_y(z)dz$ and the Sobolev-Poinc\'{a}re inequality leads to
\begin{equation}\label{2.9}\fint_{B(z_0,r)}|\nabla_z \chi^k_y(z)|^2dz\leq C\left(\fint_{B(z_0,2r)}|\nabla_z \chi^k_y(z)|^{\frac{2n}{2+n}}dz\right)^{\frac{2+n}{n}}+C.\end{equation}
Using the reverse inequality (see \cite[Chapter V, Theorem 1.2]{Giaquinta1983Multiply}), we could obtain higher integrability, and there exists a $\tau>0$,
depending on $\alpha,\beta,n$ such that
\begin{equation}\label{2.10}\fint_{B(z_0,r)}|\nabla_z \chi^k_y(z)|^{2+\tau}dz\leq C\left(\fint_{B(z_0,2r)}|\nabla_z \chi^k_y(z)|^2dz\right)^{\frac{2+\tau}{2}}+C.\end{equation} Consequently, a covering argument will give the desired estimate due to $||\chi^k_y(z)||_{W^{1,2}(Z)}\leq C$ and $\chi^k_y(z)$ is Z-periodic.
\end{proof}
In the following three lemmas, we introduce three flux correctors which will be useful for obtaining the convergence rates.
\begin{lemma}(flux corrector $E_1$)Let
\begin{equation}\label{2.11}I_{1,ij}(y,z)=-a_{ij}(y,z)+a_{ik}(y,z)\partial_{z_k}\chi^j_y(z)+\fint_Z\left(a_{ij}(y,z)-a_{ik}(y,z)\partial_{z_k}\chi^j_y(z)\right)dz,\end{equation}
where $y\in Y$ and $z\in Z$. Then there hold: $(i)$ $\fint_ZI_{1,ij}(y,\cdot)dz=0$; $(ii)$ $\partial_{z_i}I_{1,ij}=0$ for any $j=1,\cdots,n$. Moreover,
there exists the so-called flux corrector $E_{1,kij}(y,\cdot)\in H^1_{\text{per}}(Z)$ such that
\begin{equation}\label{2.12}
I_{1,ij}(y,z)=\partial_{z_k}E_{1,kij}(y,z) \text{ and } E_{1,kij}=-E_{1,ikj},
\end{equation}
and there hold  the following estimates
\begin{equation}\label{2.13}\fint_Z|E_{1,kij}(y',z)-E_{1,kij}(y,z)|^2dz+\fint_Z|\nabla_z(E_{1,kij}(y',z)-E_{1,kij}(y,z))|^2dz\leq C|y-y'|^2\end{equation}
for any $k,i,j=1,\cdots,n$, where $C$ depends on $\alpha,\beta,M$ and $n$.\end{lemma}
\begin{proof}The $(i)$ and $(ii)$ follow from the definition $(\ref{2.11})$ and $(\ref{1.7})$, respectively. By $(i)$ and $(ii)$, there exists $f_{1,ij}(y,\cdot)\in H^2(Z)$ such that
$\Delta
f_{1,ij}(y,\cdot)=I_{1,ij}(y,\cdot)$ in $Z$. Let $E_{1,kij}(y,\cdot)=\partial_{z_k}f_{1,ij}(y,\cdot)-\partial_{z_i}f_{1,kj}(y,\cdot)$, then
$E_{1,kij}=-E_{1,ikj}$ is clear, and $I_{1,ij}(y,z)=\partial_{z_k}{E_{1,kij}(y,z)}$ follows form the fact (ii). Then, for any $y,y'$, and due to the $H^2$ estimate of Laplace equation, there hold
$$\begin{aligned}
\int_{Z}|\nabla_{z}E_{1,kij}(y,\cdot)-\nabla_{z}E_{1,kij}(y',\cdot)|^2dz&\leq \int_{Z}|\nabla^2_{z}(f_{1,ij}(y,\cdot)-f_{1,ij}(y',\cdot))|^2dz\\
&\leq C\int_{2Z}|I_{1,ij}(y,z)-I_{1,ij}(y',z)|^2dz\\
&\leq C|y-y'|^2,
\end{aligned}$$ where we have used $(\ref{1.3})$ and $(\ref{2.5})$ in the last inequality. Consequently,
the estimate above together with Poincar\'{e}' inequality completes the proof of $(\ref{2.13})$.
\end{proof}

\begin{lemma}(flux corrector $E_2$)Let
\begin{equation}\label{2.14}\begin{aligned}
I_{2,ij}(y)=\widehat{a}_{ij}&+\fint_Z\left(a_{ik}(y,z)\partial_{y_k}\chi^j(y)-a_{ik}(y,z)\partial_{z_k}\chi^l_y(z)\partial_{y_l}\chi^j(y)\right)dz\\
&-\fint_Z\left(a_{ij}(y,z)-a_{ik}(y,z)\partial_{z_k}\chi^j_y(z)\right)dz,
\end{aligned}\end{equation}\end{lemma}
where $y\in Y$, we assume that $\fint_YI_{2,ij}(y)dy=0$ in addition. Then  $\partial_{y_i}I_{2,ij}=0$ for any $j=1,\cdots,n$. Moreover,
there exists the so-called flux corrector $E_{2,kij}(y)\in H^1_{\text{per}}(Y)$ such that
\begin{equation}\label{2.15}
I_{2,ij}(y)=\partial_{y_k}E_{2,kij}(y) \text{ and } E_{2,kij}=-E_{2,ikj},
\end{equation}
and the estimate \begin{equation}\label{2.16}
||E_{2,kij}||_{H^1(Y)}\leq C,
\end{equation}
for any $k,i,j=1,\cdots,n$, where $C$ depends on $\alpha,\beta,M$ and $n$.
\begin{proof}$\partial_{z_i}I_{1,ij}=0$ follows from  $(\ref{1.8})$ . By $\partial_{z_i}I_{1,ij}=0$ and the assumption $\fint_YI_{2,ij}(y)dy=0$, there exists $f_{2,ij}(y)\in H^2(Y)$ such that
$\Delta
f_{2,ij}(y)=I_{2,ij}(y)$ in $Y$. Let $E_{2,kij}(y)=\partial_{y_k}f_{2,ij}(y)-\partial_{y_i}f_{2,kj}(y)$, then
$E_{2,kij}=-E_{2,ikj}$ is clear, and $I_{2,ij}(y)=\partial_{y_k}{E_{2,kij}(y)}$ follows form the fact $\partial_{z_i}I_{2,ij}=0$.
Then according to the $H^2$ estimate of Laplace equation, there holds
$$||E_{2,kij}||_{H^1(Y)}\leq C||\nabla f_{2,ij}||_{H^1(Y)}\leq C ||I_{2,ij}||_{L^2(2Y)}\leq C,$$
where we have used $(\ref{2.2})$ in the above equality, and thus completes the proof.
\end{proof}

\begin{lemma}(flux corrector $E_3$)Let
\begin{equation}\label{2.17}\begin{aligned}
I_{3,ij}(y,z)=&a_{ik}(y,z)\partial_{y_k}\chi^j(y)-a_{ik}(y,z)\partial_{z_k}\chi^l_y(z)\partial_{y_l}\chi^j(y)\\
&-\fint_Z\left(a_{ik}(y,z)\partial_{y_k}\chi^j(y)-a_{ik}(y,z)\partial_{z_k}\chi^l_y(z)\partial_{y_l}\chi^j(y)\right)dz,
\end{aligned}\end{equation}
where $y\in Y$ and $z\in Z$.Then there hold: $(i)$ $\fint_ZI_{3,ij}(y,\cdot)dz=0$; $(ii)$ $\partial_{z_i}I_{3,ij}=0$ for any $j=1,\cdots,n$. Moreover,
there exists the so-called flux corrector $E_{3,ij}(y,\cdot)\in H^1_{\text{per}}(Z)$ such that
\begin{equation}\label{2.18}
I_{3,ij}(y,z)=\partial_{z_k}E_{3,kij}(y,z) \text{ and } E_{3,kij}=-E_{3,ikj},
\end{equation}
for any $i,j=1,\cdots,n$.\end{lemma}
\begin{proof}The $(i)$ and $(ii)$ follow from the definition $(\ref{2.17})$ and $(\ref{1.8})$, respectively. By $(i)$ and $(ii)$, there exists $f_{3,ij}(y,\cdot)\in H^2(Z)$ such that
$\Delta
f_{3,ij}(y,\cdot)=I_{3,ij}(y,\cdot)$ in $Z$. Let $E_{3,kij}(y,\cdot)=\partial_{z_k}f_{3,ij}(y,\cdot)-\partial_{z_i}f_{3,kj}(y,\cdot)$. Similarly to the reasons in Lemma 2.3, we complete the proof.
\end{proof}

To deal with the convergence rates in the next section, we introduce an $\varepsilon$-smoothing operator $S_\varepsilon$.\\
\begin{defn} Fix a nonnegative function $\rho\in C_0^\infty(B(0,1/2))$ such that $\int_{\mathbb{R}^n}\rho dx=1$. For $\varepsilon>0,$
define \begin{equation}\label{2.19}
S_\varepsilon(f)(x)=\rho_\varepsilon\ast f(x)=\int_{\mathbb{R}^n}f(x-y)\rho_\varepsilon(y)dy,\end{equation}
where $\rho_\varepsilon(y)=\varepsilon^{-n}\rho(y/\varepsilon)$. \end{defn}
\begin{lemma} (i) If $f\in L^p(\mathbb{R}^n)$ for some $1\leq p<\infty.$ Then for any $g\in L^{p}_{\text{per}}(\mathbb{R}^n)$ ($g$ is $Y$-periodic),
\begin{equation}\label{2.20}\begin{cases}
||g(\cdot/\varepsilon)S_\varepsilon(f)||_{L^p(\mathbb{R}^n)}\leq C(p,n)||g||_{L^p(Y)}||f||_{L^p(\mathbb{R}^n)}\\
||g(\cdot/\varepsilon)\nabla S_\varepsilon(f)||_{L^p(\mathbb{R}^n)}\leq C(p,n)\varepsilon^{-1}||g||_{L^p(Y)}||f||_{L^p(\mathbb{R}^n)},
\end{cases}\end{equation} and if $0<\varepsilon\leq 1$, we have
\begin{equation}\label{2.21}
||g(\cdot/\varepsilon^2)S_\varepsilon(f)||_{L^p(\mathbb{R}^n)}\leq C(p,n)||g||_{L^p(Y)}||f||_{L^p(\mathbb{R}^n)}.
\end{equation}
(ii) If $f\in W^{1,p}(\mathbb{R}^n)$ for some $1\leq p<\infty.$ Then
\begin{equation}\label{2.22}
||S_\varepsilon(f)-f||_{L^p(\mathbb{R}^n)}\leq C(n,p)\varepsilon||\nabla f||_{L^p(\mathbb{R}^n)}.
\end{equation}
(iii) If $g\in H^1(\mathbb{R}^{n})$, then
\begin{equation}\label{2.23}\begin{cases}
||S_\varepsilon(g)-g||_{{L^2}(\mathbb{R}^{n})}\leq C \varepsilon||g||_{{\dot{H}^{1}}(\mathbb{R}^{n})}\\
||S_\varepsilon(g)||_{{\dot{H}^{1/2}}(\mathbb{R}^{n})}\leq C||g||_{{\dot{H}^{1/2}}(\mathbb{R}^{n})}\\
||S_\varepsilon(g)-g||_{{\dot{H}^{1/2}}(\mathbb{R}^{n})}\leq C \varepsilon^{1/2}||g||_{{\dot{H}^{1}}(\mathbb{R}^{n})}\\
||S_\varepsilon(g)||_{{\dot{H}^{3/2}}(\mathbb{R}^{n})}\leq C \varepsilon^{-1/2}||g||_{{\dot{H}^{1}}(\mathbb{R}^{n})}\\
\end{cases}\end{equation}
 in which the constant C depends on $n$.\end{lemma}

\begin{proof}For the proof of (i), see for example \cite[Proposition 3.1.5]{shen2018periodic}, for the proof of (ii) and (iii), see for example \cite[Proposition 3.1.6]{shen2018periodic}. Therefore, we need only give the proof of $(\ref{2.21})$. By H\"{o}lder's inequality,
$$\left|S_\varepsilon(f)(x)\right|^p\leq\int_{\mathbb{R}^n}\left|f(y)\right|^p\rho_{\varepsilon}(x-y)dy.$$
This together with Fubini's Theorem, gives
\begin{equation}\label{2.24}\begin{aligned}
\int_{\mathbb{R}^n}\left|g(x/\varepsilon^2)\right|^p\left|S_\varepsilon(f)(x)\right|^pdx&\leq\iint_{\mathbb{R}^n\times
\mathbb{R}^n}\left|g(x/\varepsilon^2)\right|^p\left|f(y)\right|^p\rho_{\varepsilon}(x-y)dx dy\\
&\leq C \sup_{y\in\mathbb{R}^n}\varepsilon^{-n}\int_{|x-y|\leq\varepsilon/2}\left|g(x/\varepsilon^2)\right|^pdx||f||_{L^p(\mathbb{R}^n)}^p\\
&\leq C \sup_{y\in\mathbb{R}^n}\varepsilon^{n}\int_{|x-y|\leq1/(2\varepsilon)}\left|g(x)\right|^pdx||f||_{L^p(\mathbb{R}^n)}^p\\
&\leq C ||f||_{L^p(\mathbb{R}^n)}^p||g||_{L^p(Y)}^p, \end{aligned}\end{equation}
where we use the periodicity of $g$ and note that $0<\varepsilon\leq1$.
\end{proof}
\begin{rmk}Actually, under the assumption of Lemma 2.7 (i), if $0<\varepsilon\leq 1$, for any $\lambda\geq \mu>0$, there holds
\begin{equation}\label{2.25}
||g(\cdot/\varepsilon^\lambda)S_{\varepsilon^\mu}(f)||_{L^p(\mathbb{R}^n)}\leq C(p,n)||g||_{L^p(Y)}||f||_{L^p(\mathbb{R}^n)}.
\end{equation} However, the similar results couldn't hold for the function $g(\cdot/\varepsilon^\lambda)S_{\varepsilon^\mu}(f)$,
if $0<\lambda< \mu$, unless the function $g$ has better regularity.
\end{rmk}
\section{Convergence rates}
First of all, we introduce the following cut-off function $\psi_r\in C_0^1(\Omega)$ associated with $\Sigma_r$:
\begin{equation}\label{3.1}
\psi_r=1\ \ \text{in\ }\Sigma_{2r},\ \ \ \psi_r=0\ \ \text{outside\ }\Sigma_{r},\ \ \ |\nabla\psi_r|\leq C/r,
\end{equation}
where $\Sigma_r=\{x\in\Omega:\text{dist}(x,\partial\Omega)>r\}.$

\begin{lemma}Let $\Omega$ be a bounded Lipschitz domain, and assume that $\mathcal{L}_\varepsilon$ satisfies the assumptions $(\ref{1.2}),\ (\ref{1.3})$
and $(\ref{1.4})$. Suppose that $u_\varepsilon,u_0\in H^1(\Omega)$ satisfy $\mathcal{L}_\varepsilon u_\varepsilon=\mathcal{L}_0 u_0$ in $\Omega$, and the first-order approximating corrector is given by
\begin{equation}\label{3.2}\begin{aligned}
w_\varepsilon(x)=&u_\varepsilon(x)-u_0(x)+\varepsilon \chi^j(x/\varepsilon)\psi_{2\varepsilon}S_\varepsilon\left(\partial_{x_j}u_0\right)\\
&+\varepsilon^2\chi^j_y(x/\varepsilon^2)\left[\psi_{2\varepsilon}S_\varepsilon\left(\partial_{x_j}u_0\right)
-\partial_{y_j}\chi^{k}(x/\varepsilon)\psi_{2\varepsilon}S_\varepsilon\left(\partial_{x_k}u_0\right)\right],
\end{aligned}\end{equation} where $\chi^k_y(x/\varepsilon^2)=\chi^k(x/\varepsilon,x/\varepsilon^2)$.
Then for any $\phi\in H^1_0(\Omega)$, it gives
\begin{equation}\label{3.3}\begin{aligned}
\left|\int_\Omega a_{ij}(x/\varepsilon,x/\varepsilon^2)\partial_jw_\varepsilon\partial_i\phi dx\right|
\leq C\varepsilon&\left(||\nabla u_0||_{L^2(\Omega)}+||\nabla^2 u_0||_{L^2(\Omega)}\right)||\nabla \phi||_{L^2(\Omega)}\\
&+C||\nabla u_0||_{L^2(\Omega\setminus\Sigma_{5\varepsilon})}||\nabla \phi||_{L^2(\Omega\setminus\Sigma_{4\varepsilon})},
\end{aligned}\end{equation}
where $C$ depends on $\alpha,\beta,M$ and $n$.

\end{lemma}
\begin{proof}By direct computation,  we have
\begin{equation}\label{3.4}\begin{aligned}
&a_{ih}(x/\varepsilon,x/\varepsilon^2)\partial_hw_\varepsilon\\
=&a_{ih}(\partial_hu_\varepsilon-\partial_hu_0)+a_{ih}\partial_{y_h}\chi^j(x/\varepsilon)\psi_{2\varepsilon}S_\varepsilon\left(\partial_{x_j}u_0\right)
+\varepsilon a_{ih} \chi^j(x/\varepsilon)\partial_h\left(\psi_{2\varepsilon}S_\varepsilon\left(\partial_{x_j}u_0\right)\right)\\
&+a_{ih}\partial_{z_h}\chi^j_y(x/\varepsilon^2)\left[\psi_{2\varepsilon}S_\varepsilon\left(\partial_{x_j}u_0\right)-\partial_{y_j}\chi^{k}(x/\varepsilon)
\psi_{2\varepsilon}S_\varepsilon\left(\partial_{x_k}u_0\right)\right]\\
&+\varepsilon a_{ih}\partial_{y_h}\chi^j_y(x/\varepsilon^2)\left[\psi_{2\varepsilon}S_\varepsilon\left(\partial_{x_j}u_0\right)-\partial_{y_j}\chi^{k}(x/\varepsilon)
\psi_{2\varepsilon}S_\varepsilon\left(\partial_{x_k}u_0\right)\right]\\
&+\varepsilon^2a_{ih}\chi^j_y(x/\varepsilon^2)\left[\partial_h\left(\psi_{2\varepsilon}S_\varepsilon\left(\partial_{x_j}u_0\right)\right)
-\partial_{y_j}\chi^{k}(x/\varepsilon)\partial_h\left(\psi_{2\varepsilon}S_\varepsilon\left(\partial_{x_k}u_0\right)\right)\right]\\
&-\varepsilon a_{ih}\chi^j_y(x/\varepsilon^2)\partial^2_{y_jy_h}\chi^{k}(x/\varepsilon)\psi_{2\varepsilon}S_\varepsilon\left(\partial_{x_k}u_0\right)\\
=&:H_{1,i}+H_{2,i}+H_{3,i}+H_{4,i}
\end{aligned}\end{equation}
and \begin{equation}\label{3.5}\begin{aligned}H_{1,i}=&(\widehat{a}_{ih}-a_{ih})(\partial_hu_0-\psi_{2\varepsilon}S_\varepsilon\left(\partial_{x_h}u_0\right))
+a_{ih}\partial_hu_\varepsilon-\widehat{a}_{ih}\partial_hu_0,\\
H_{2,i}=&[\widehat{a}_{ij}-a_{ij}+a_{ih}\partial_{y_h}\chi^j(x/\varepsilon)+a_{ih}\partial_{z_h}\chi^j_y(x/\varepsilon^2)
-a_{ih}\partial_{z_h}\chi^k_y(x/\varepsilon^2)\partial_{y_k}\chi^{j}(x/\varepsilon)]\psi_{2\varepsilon}S_\varepsilon\left(\partial_{j}u_0\right)\\
H_{3,i}=&\varepsilon a_{ih}\chi^j(x/\varepsilon)\partial_h\left(\psi_{2\varepsilon}S_\varepsilon\left(\partial_{x_j}u_0\right)\right)-\varepsilon a_{ih}\chi^j_y(x/\varepsilon^2)\partial^2_{y_jy_h}\chi^{k}(x/\varepsilon)\psi_{2\varepsilon}S_\varepsilon\left(\partial_{k}u_0\right)\\
&+\varepsilon a_{ih}\partial_{y_h}\chi^j_y(x/\varepsilon^2)\left[\psi_{2\varepsilon}S_\varepsilon\left(\partial_{x_j}u_0\right)
-\partial_{y_j}\chi^{k}(x/\varepsilon)\psi_{2\varepsilon}S_\varepsilon\left(\partial_{x_k}u_0\right)\right],\\
H_{4,i}=&\varepsilon^2a_{ih}\chi^j_y(x/\varepsilon^2)\left[\partial_h\left(\psi_{2\varepsilon}S_\varepsilon\left(\partial_{x_j}u_0\right)\right)
-\partial_{y_j}\chi^{k}(x/\varepsilon)\partial_h\left(\psi_{2\varepsilon}S_\varepsilon\left(\partial_{x_k}u_0\right)\right)\right].
\end{aligned}\end{equation}
To estimate $H_{2,i}$, we observe that
$$H_{2,i}=(I_{1,ij}(x/\varepsilon,x/\varepsilon^2)+I_{2,ij}(x/\varepsilon)+I_{3,ij}(x/\varepsilon,x/\varepsilon^2))\psi_{2\varepsilon}S_\varepsilon(\partial_{j}u_0)$$ with $I_{1,ij}$, $I_{2,ij}$ and $I_{3,ij}$ defined in $(\ref{2.11})$, $(\ref{2.14})$ and $(\ref{2.17})$, respectively.
Then \begin{equation}\label{3.6}\left|\int_\Omega H_{2,i}\partial_i \phi dx\right|=\left|\int_\Omega (I_{1,ij}(x/\varepsilon,x/\varepsilon^2)+I_{2,ij}(x/\varepsilon)+I_{3,ij}(x/\varepsilon,x/\varepsilon^2))\psi_{2\varepsilon}S_\varepsilon(\partial_{j}u_0)\partial_i \phi dx\right|.
\end{equation}

According to $(\ref{1.6})$, we have $\iint_{Y\times Z}I_{1,ij}(y,z)+I_{2,ij}(y)+I_{3,ij}(y,z)dydz=0$, then $\iint_{Y\times Z}I_{2,ij}(y)dydz=0$ due to Lemma 2.3 (i) and Lemma 2.5 (i). Consequently, we have
$\int_{Y}I_{2,ij}(y)dy=0$ which satisfies the assumption of Lemma 2.4.\\

However, the estimates of the term
\begin{equation}\label{3.7}W=:\left|\int_\Omega (I_{1,ij}(x/\varepsilon,x/\varepsilon^2)+I_{3,ij}(x/\varepsilon,x/\varepsilon^2))\psi_{2\varepsilon}S_\varepsilon(\partial_{j}u_0)\partial_i \phi dx\right|=:|W_1+W_2|\end{equation} requires more technical skills, and we will give this estimate in the following proposition.

\begin{pro}Let $W$ be defined by $(\ref{3.7})$, then there holds
\begin{equation}\label{3.8}W\leq C\varepsilon||\nabla u_0||_{L^2(\Omega)}||\nabla \phi||_{L^2(\Omega)}+C\varepsilon^2||\nabla^2 u_0||_{L^2(\Omega)}||\nabla \phi||_{L^2(\Omega)}.\end{equation}
\end{pro}
\begin{proof}The difficulty of the estimate $(3.8)$ is to handle different scales $y\ (=x/\varepsilon)$ and $z\ (=x/\varepsilon^2)$ in the estimate of $W$, so the main idea is to separate
these scales by taking Fourier transform with respect to $z$, and then an important observation is that the Fourier coefficients are $Y$-periodic.
Consequently, Lemma 2.7 (i) may be adapted and thus completes this proof. We only give the proof of $W_2$, since the estimate of $W_1$ is totally the same to $W_2$. Now, we give the details. Recall we assume that $Y=Z=(0,1)^n$.

Taking the Fourier transform of $I_{3,ij}(y,z)$ with respect to $z$ gives that
\begin{equation}\label{3.9}I_{3,ij}(y,z)=\sum_{k\in\mathbb{Z}^n}\widehat{I}_{3,kij}(y)e^{2\pi\sqrt{-1} kz},\end{equation}
where $\widehat{I}_{3,kij}$ is given by
\begin{equation}\label{3.10}\widehat{I}_{3,kij}(y)=\int_{(0,1)^n}I_{3,ij}(y,z)e^{-2\pi\sqrt{-1} kz}dz.\end{equation}
According to Remark 1.1, $\widehat{I}_{3,kij}(y)$ is Y-periodic.
Clearly, as the notations in Lemma 2.5, we have
$$f_{3,ij}(y,z)=-\frac{1}{4\pi^2}\sum_{k\in\mathbb{Z}^n}|k|^{-2}\widehat{I}_{3,kij}(y)e^{2\pi\sqrt{-1} kz},$$
and
\begin{equation}\label{3.11}\begin{aligned}
E_{3,hij}(y,z)&=\partial_{z_h}f_{3,ij}(y,z)-\partial_{z_i}f_{3,hj}(y,z)\\
&=\frac{\sqrt{-1}}{2\pi}\sum_{k\in\mathbb{Z}^n}k_i|k|^{-2}\widehat{I}_{3,khj}(y)e^{2\pi\sqrt{-1} kz}-
\frac{\sqrt{-1}}{2\pi}\sum_{k\in\mathbb{Z}^n}k_h|k|^{-2}\widehat{I}_{3,kij}(y)e^{2\pi\sqrt{-1} kz}\\
&=:E_{31,hij}(y,z)+E_{32,hij}(y,z),
\end{aligned}\end{equation}
then according to $(\ref{2.18})$ and $(\ref{3.7})$, we have
\begin{equation}\label{3.12}\begin{aligned}
\left|W_2\right|=&\left|\int_\Omega \partial_{z_h}E_{3,hij}(x/\varepsilon,x/\varepsilon^2)\psi_{2\varepsilon}S_\varepsilon(\partial_{j}u_0)\partial_i \phi dx\right|\\
\leq& \varepsilon^2\left|\int_\Omega \partial_{x_h}E_{3,hij}(x/\varepsilon,x/\varepsilon^2)\psi_{2\varepsilon}S_\varepsilon(\partial_{j}u_0)\partial_i \phi dx\right|\\
&+\varepsilon\left|\int_\Omega \partial_{y_h}E_{3,hij}(x/\varepsilon,x/\varepsilon^2)\psi_{2\varepsilon}S_\varepsilon(\partial_{j}u_0)\partial_i \phi dx\right|\\
=&\varepsilon^2\left|\int_\Omega(E_{31,hij}+E_{32,hij})(x/\varepsilon,x/\varepsilon^2)\partial_{x_h}\left(\psi_{2\varepsilon}S_\varepsilon(\partial_{j}u_0)\right)\partial_i \phi dx\right|\\
&+\varepsilon\left|\int_\Omega\partial_{y_h}(E_{31,hij}+E_{32,hij})(x/\varepsilon,x/\varepsilon^2)\psi_{2\varepsilon}S_\varepsilon(\partial_{j}u_0)\partial_i \phi dx\right|\\
=&:\varepsilon^2|W_{21}+W_{22}|+\varepsilon|W_{23}+W_{24}|.
\end{aligned}\end{equation}

Consequently, we need only to estimate $|W_{21}|$ and $|W_{23}|$, since the estimates of $W_{22}$ is the same to $W_{21}$, and  $W_{24}$ is the same to $W_{23}$, respectively. We firstly treat $\varepsilon^2|W_{21}|$.
In view of $(\ref{3.11})$ and Lemma 2.7 (i), there holds
\begin{equation}\label{3.13}\begin{aligned}
\varepsilon^2|W_{21}|&\leq\varepsilon^2\int_{\Omega}\left|\sum_{k\in\mathbb{Z}^n}k_i|k|^{-2}\widehat{I}_{3,khj}(x/\varepsilon)\partial_{x_h}
\left(\psi_{2\varepsilon}S_\varepsilon(\partial_{j}u_0)\right)\partial_i \phi\right| dx\\
&\leq C||\nabla \phi||_{L^2(\Omega)}\left(\varepsilon^2||\nabla^2u_0||_{L^2(\Omega)}+\varepsilon||\nabla u_0||_{L^2(\Omega)}\right)\left(\int_{Y}\left|\sum_{k\in\mathbb{Z}^n}k_i|k|^{-2}\widehat{I}_{3,khj}(y)\right|^2dy\right)^{1/2}
\end{aligned}\end{equation}
H\"{o}lder's inequality and Plancherel's Indetity  give that
\begin{equation}\label{3.14}\begin{aligned}\int_{Y}\left|\sum_{k\in\mathbb{Z}^n}k_i|k|^{-2}\widehat{I}_{3,khj}(y)\right|^2dy&\leq
\int_{Y}\sum_{k\in\mathbb{Z}^n}\left|\widehat{I}_{3,khj}(y)\right|^2dy\cdot \sum_{k\in\mathbb{Z}^n}\left|k_i|k|^{-2}\right|^2\\
&\leq C\iint_{Y\times Z}\left|I_{3,hj}(y,z)\right|^2dzdy\\
&\leq C,
\end{aligned}\end{equation} due to $(\ref{2.2})$ and Sobolev embedding inequality. Consequently, combining $(\ref{3.13})$ and $(\ref{3.14})$ gives
\begin{equation}\label{3.15}\varepsilon^2|W_{21}|\leq C||\nabla \phi||_{L^2(\Omega)}\left(\varepsilon^2||\nabla^2u_0||_{L^2(\Omega)}+\varepsilon||\nabla u_0||_{L^2(\Omega)}\right).\end{equation}
As for $W_{23}$, we have
\begin{equation}\label{3.16}\begin{aligned}
\varepsilon|W_{23}|&\leq\varepsilon\int_{\Omega}\left|\sum_{k\in\mathbb{Z}^n}k_i|k|^{-2}\partial_{y_h}\widehat{I}_{3,khj}(x/\varepsilon)
\psi_{2\varepsilon}S_\varepsilon(\partial_{j}u_0)\partial_i \phi\right| dx\\
&\leq C\varepsilon^2||\nabla \phi||_{L^2(\Omega)}||\nabla u_0||_{L^2(\Omega)}\left(\int_{Y}\left|\sum_{k\in\mathbb{Z}^n}k_i|k|^{-2}\partial_{y_h}\widehat{I}_{3,khj}(y)\right|^2dy\right)^{1/2},
\end{aligned}\end{equation}similarly, H\"{o}lder's inequality and Plancherel's Indetity  give that

\begin{equation}\label{3.17}\begin{aligned}\int_{Y}\left|\sum_{k\in\mathbb{Z}^n}k_i|k|^{-2}\partial_{y_h}\widehat{I}_{3,khj}(y)\right|^2dy&\leq
\int_{Y}\sum_{k\in\mathbb{Z}^n}\left|\nabla_y\widehat{I}_{3,khj}(y)\right|^2dy\cdot \sum_{k\in\mathbb{Z}^n}\left|k_i|k|^{-2}\right|^2\\
&\leq C\iint_{Y\times Z}\left|\nabla_yI_{3,hj}(y,z)\right|^2dzdy,
\end{aligned}\end{equation}
 in view of $(\ref{2.17})$,
\begin{equation}\label{3.18}\begin{aligned}
I_{3,ij}(y,z)=&a_{ik}(y,z)\partial_{y_k}\chi^j(y)-a_{ik}(y,z)\partial_{z_k}\chi^l_y(z)\partial_{y_l}\chi^j(y)\\
&-\fint_Z\left(a_{ik}(y,z)\partial_{y_k}\chi^j(y)-a_{ik}(y,z)\partial_{z_k}\chi^l_y(z)\partial_{y_l}\chi^j(y)\right)dz\\
=&:I_{31,ij}(y,z)+I_{32,ij}(y,z)+I_{33,ij}(y,z),
\end{aligned}\end{equation}
then according $(\ref{1.3})$ to $(\ref{2.2})$, it gives
\begin{equation}\label{3.19}\iint_{Y\times Z}\left|\nabla_yI_{31,ij}(y,z)\right|^2dzdy\leq C.\end{equation}
And
\begin{equation}\label{3.20}\begin{aligned}
||\nabla_yI_{32,ij}(y,z)||_{L^2(Y\times Z)}\leq& ||\nabla_y a_{ij}(y,z)||_{L^\infty(Y\times Z)}||\nabla_y \chi^k(y)||_{L^\infty(Y)}||\nabla_z \chi^k_y(z)||_{L^2(Y\times Z)}\\
&+||a_{ij}(y,z)||_{L^\infty(Y\times Z)}||\nabla_y \chi^k(y)||_{L^\infty(Y)}||\nabla_y\nabla_z \chi^k_y(z)||_{L^2(Y\times Z)}\\
&+||a_{ij}(y,z)||_{L^\infty(Y\times Z)}||\nabla^2_y \chi^k(y)||_{L^{(4+2\tau)/\tau}(Y)}||\nabla_z \chi^k_y(z)||_{L^{2+\tau}(Y\times Z)}\\
\leq& C,
\end{aligned}\end{equation} where we have used $(\ref{1.3})$, $(\ref{2.1})$, $(\ref{2.2})$, $(\ref{2.7})$ and the Sobolev embedding inequality in $(\ref{3.20})$. Clearly,
\begin{equation}\label{3.21}||\nabla_y I_{33,ij}(y,z)||_{L^2(Y\times Z)}\leq C\end{equation}
directly follows from $(\ref{3.19})$ and $(\ref{3.20})$. Consequently, combining $(\ref{3.15})-(\ref{3.21})$ leads to the conclusion $(\ref{3.8})$.
\end{proof}

After obtaining the estimate of $W$, we now continue the proof of Lemma 3.1. It is not hard to see that
\begin{equation}\label{3.22}\begin{aligned}\left|\int_\Omega
I_{2,ij}\psi_{2\varepsilon}S_\varepsilon(\partial_{j}u_0)\partial_i \phi dx\right|&=\left|\int_\Omega
\partial_{y_h}E_{2,hij}(x/\varepsilon)\psi_{2\varepsilon}S_\varepsilon(\partial_{j}u_0)\partial_i \phi dx\right|\\
&=\varepsilon\left|\int_\Omega \partial_{x_h}E_{2,hij}(x/\varepsilon)\psi_{2\varepsilon}S_\varepsilon(\partial_{j}u_0)\partial_i \phi
dx\right|\\
&\leq\varepsilon\left|\int_\Omega E_{2,hij}(x/\varepsilon)\partial_{x_h}(\psi_{2\varepsilon}S_\varepsilon(\partial_{j}u_0))\partial_i \phi dx\right|\\
&\leq\varepsilon||\nabla^2 u_0||_{L^2(\Omega)}||\nabla
\phi||_{L^2(\Omega)}+C||\nabla u_0||_{L^2(\Omega\setminus\Sigma_{5\varepsilon})}||\nabla
\phi||_{L^2(\Omega\setminus\Sigma_{4\varepsilon})},\end{aligned}\end{equation} where we have used Lemma 2.7 (i) and $(\ref{2.16})$ in the last inequality in $(\ref{3.22})$.

Therefore, combining the estimates $(\ref{3.8})$ and $(\ref{3.22})$ gives that
\begin{equation}\label{3.23}\begin{aligned}
\left|\int_\Omega H_{2,i}\partial_i \phi dx\right|\leq C\varepsilon&\left(||\nabla u_0||_{L^2(\Omega)}+||\nabla^2 u_0||_{L^2(\Omega)}\right)||\nabla \phi||_{L^2(\Omega)}\\
&+||\nabla u_0||_{L^2(\Omega\setminus\Sigma_{5\varepsilon})}||\nabla \phi||_{L^2(\Omega\setminus\Sigma_{4\varepsilon})}.\end{aligned}\end{equation}
Due to $\mathcal{L}_\varepsilon u_\varepsilon=\mathcal{L}_0 u_0$ in $\Omega$ and $\phi\in H^1_0(\Omega)$, it gives that
\begin{equation}\label{3.24}\begin{aligned}
\left|\int_\Omega H_{1,i}\partial_i \phi dx\right|&=\left|\int_\Omega(\widehat{a}_{ih}-a_{ih})(\partial_hu_0
-\psi_{2\varepsilon}S_\varepsilon\left(\partial_{h}u_0\right))\partial_i \phi dx\right|\\
&\leq C\int_\Omega\left|\nabla u_0-S_\varepsilon\left(\nabla u_0\right)\right|\left|\nabla \phi \right|dx+C\int_\Omega\left| S_\varepsilon(\nabla u_0)-\psi_{2\varepsilon}S_\varepsilon(\nabla u_0)\right|\left|\nabla \phi \right|dx\\
&\leq C\varepsilon||\nabla^2 u_0||_{L^2(\Omega)}||\nabla \phi||_{L^2(\Omega)}
+C||\nabla u_0||_{L^2(\Omega\setminus\Sigma_{5\varepsilon})}||\nabla \phi||_{L^2(\Omega\setminus\Sigma_{4\varepsilon})},
\end{aligned}\end{equation} where we have used the Lemma 2.7 (ii) in the last inequality. \\

In view of the definition of the $H_{3,i}$ and $H_{4,i}$ in $(3.5)$,
in order to estimate the term $\left|\int(H_{3,i}+H_{4,i})\partial_i \phi\right|$, we need only to estimate $\left|\int\varepsilon
a_{ih}\partial_{y_h}\chi^j_y(x/\varepsilon^2)\psi_{2\varepsilon}S_\varepsilon\left(\partial_{x_j}u_0\right)\partial_i \phi\right|$. Similar to the
computation in Proposition 3.2, taking the Fourier transform of $\chi^j(y,z)$ with respect to $z$ leads to
\begin{equation*}\chi^j(y,z)=\sum_{k\in\mathbb{Z}^n}\widehat{\chi}^j_{k}(y)e^{2\pi\sqrt{-1} kz},\end{equation*}
where $\widehat{\chi}^j_{k}(y)$ is given by
\begin{equation*}\widehat{\chi}^j_{k}(y)=\int_{(0,1)^n}\chi^j(y,z)e^{-2\pi\sqrt{-1} kz}dz.\end{equation*}
Then according to Lemma 2.7 (i), we have
\begin{equation*}\begin{aligned}
&\left|\int_\Omega\varepsilon a_{ih}\partial_{y_h}\chi^j_y(x/\varepsilon^2)\psi_{2\varepsilon}S_\varepsilon\left(\partial_{x_j}u_0\right)\partial_i \phi dx\right|\\
\leq&\varepsilon\int_\Omega \left| a_{ih}\sum_{k\in\mathbb{Z}^n}\partial_{y_h}\widehat{\chi}^j_{k}(x/\varepsilon)\psi_{2\varepsilon}S_\varepsilon\left(\partial_{x_j}u_0\right)\partial_i \phi\right| dx\\
\leq& C\varepsilon^2||\nabla \phi||_{L^2(\Omega)}||\nabla u_0||_{L^2(\Omega)}\left(\int_{Y}\left|\sum_{k\in\mathbb{Z}^n}\partial_{y_h}\widehat{\chi}^j_{k}(y)\right|^2dy\right)^{1/2},
\end{aligned}\end{equation*} and H\"{o}lder's inequality and Plancherel's Indetity  give that
\begin{equation*}\begin{aligned}
\int_{Y}\left|\sum_{k\in\mathbb{Z}^n}\partial_{y_h}\widehat{\chi}^j_{k}(y)\right|^2dy&\leq \sum_{k\in\mathbb{Z}^n} |k|^{-2}\int_{Y}\sum_{k\in\mathbb{Z}^n}|k|^2\left|\nabla_y\widehat{\chi}^j_{k}(y)\right|^2dy\\
&\leq C\iint_{Y\times Z}\left|\nabla_z\nabla_y\chi^j(y,z)\right|^2dydz\\
& \leq C,
\end{aligned}\end{equation*} where we use $(2.1)$ in the above inequality. \\

Consequently, according to $(\ref{2.1})$, $(\ref{2.2})$, Lemma 2.7 (i) and (ii) as well as the Sobolev embedding inequality, we can easily have
\begin{equation}\label{3.25}\begin{aligned}
\left|\int_\Omega H_{3,i}\partial_i \phi dx\right|\leq C\varepsilon&\left(||\nabla u_0||_{L^2(\Omega)}+||\nabla^2 u_0||_{L^2(\Omega)}\right)||\nabla
\phi||_{L^2(\Omega)}\\
&+||\nabla u_0||_{L^2(\Omega\setminus\Sigma_{5\varepsilon})}||\nabla \phi||_{L^2(\Omega\setminus\Sigma_{4\varepsilon})}.
\end{aligned}\end{equation} and
\begin{equation}\label{3.26}
\left|\int_\Omega H_{4,i}\partial_i \phi dx\right|\leq C\left(\varepsilon||\nabla u_0||_{L^2(\Omega)}+\varepsilon^2||\nabla^2 u_0||_{L^2(\Omega)}\right)||\nabla \phi||_{L^2(\Omega)},
\end{equation}consequently, combining  the estimates $(\ref{3.23})-(\ref{3.26})$ gives the desired result.
\end{proof}

\begin{rmk}In order to obtain better estimates, if $w_\varepsilon$ has the form
\begin{equation}\label{3.27}\begin{aligned}
w_\varepsilon(x)=&u_\varepsilon(x)-u_0(x)+\varepsilon \chi^j(x/\varepsilon)\psi_{2\varepsilon^\lambda}S_{\varepsilon^\lambda}\left(\partial_{x_j}u_0\right)\\
&+\varepsilon^2\chi^j_y(x/\varepsilon^2)\left[\psi_{2\varepsilon^\lambda}S_{\varepsilon^\lambda}
\left(\partial_{x_j}u_0\right)-\partial_{y_j}\chi^{k}(x/\varepsilon)\psi_{2\varepsilon^\lambda}S_{\varepsilon^\lambda}\left(\partial_{x_k}u_0\right)\right],
\end{aligned}\end{equation}
where $0<\lambda$ is a constant which is to be chosen. In view of Remark 2.8, we need to assume $\lambda\leq 2$. (Actually, in view of the term
$a_{ih}\partial_{y_h}\chi^j_y(x/\varepsilon^2)\partial_{y_j}\chi^{k}(x/\varepsilon)\psi_{2\varepsilon}S_\varepsilon\left(\partial_{x_k}u_0\right)$ in $H_{3,i}$ defined in $(\ref{3.5})$, we need $\lambda \leq 2$.  However, if $\lambda>1$, we need more regularity assumptions on $\chi^k(y)$ and $\chi^k_y(z)$). Consequently, careful computation shows that $\lambda=1$ is the best choice, which may declare that the scale of $\varepsilon$ dominates any other scales. The same
result holds for $w_\varepsilon$ of the form
\begin{equation*}\begin{aligned}
w_\varepsilon(x)=&u_\varepsilon(x)-u_0(x)+\varepsilon \chi^j(x/\varepsilon)\psi_{2\varepsilon^\mu}S_{\varepsilon^\lambda}\left(\partial_{x_j}u_0\right)\\
&+\varepsilon^2\chi^j_y(x/\varepsilon^2)\left[\psi_{2\varepsilon^\mu}S_{\varepsilon^\lambda}
\left(\partial_{x_j}u_0\right)-\partial_{y_j}\chi^{k}(x/\varepsilon)\psi_{2\varepsilon^\mu}S_{\varepsilon^\lambda}\left(\partial_{x_k}u_0\right)\right],
\end{aligned}\end{equation*} where $0<\mu<\lambda\leq 2$.
\end{rmk}

\begin{lemma} Assume the same conditions as in Lemma 3.1, and $u_\varepsilon=u_0$ on $\partial\Omega$.  Then there hold the following estimates
\begin{equation}\label{3.28}
||\nabla w_\varepsilon||_{L^2(\Omega)}\leq C\left(\varepsilon||\nabla u_0||_{L^2(\Omega)}+||\nabla
u_0||_{L^2(\Omega\setminus\Sigma_{5\varepsilon})}+\varepsilon||\nabla^2 u_0||_{L^2(\Omega)}\right)
\end{equation}
and
\begin{equation}\label{3.29} ||u_\varepsilon-u_0||_{L^2(\Omega)}\leq C r_0\left(\varepsilon||\nabla u_0||_{L^2(\Omega)}+||\nabla
u_0||_{L^2(\Omega\setminus\Sigma_{5\varepsilon})}+\varepsilon||\nabla^2 u_0||_{L^2(\Omega)}\right).\end{equation}
where $C$ depends on $\alpha,\beta,M,n$ and the character of $\Omega$, and $r_0=\text{diam}(\Omega)$.
\end{lemma}

\begin{proof}
Due to $(\ref{3.1})$ and $u_\varepsilon=u_0$ on $\partial\Omega$,  $w_\varepsilon\in H^1_0(\Omega)$ is easy to verify. Then,  taking $\phi=w_\varepsilon$
in the Lemma 3.1, it gives the estimate $(\ref{3.28})$. Imitating the computation in Lemma 3.1 after in view of the definition $(\ref{3.2})$ of
$w_\varepsilon$
(actually, when estimating the term $||\nabla_y\chi^j(x/\varepsilon,x/\varepsilon^2)\psi_{2\varepsilon}S_\varepsilon(\partial_ju_0)||_{L^2(\Omega)}$, we take the Fourier transform of $\nabla_y\chi^j(y,z)$ with respect to $z$; and when estimating the term
$||\nabla_z\chi^j(x/\varepsilon,x/\varepsilon^2)\psi_{2\varepsilon}S_\varepsilon(\partial_ju_0)||_{L^2(\Omega)}$, we take the Fourier transform of
$\nabla_z\chi^j(y,z)$ with respect to $y$, then Plancherel's Indetity and $(2.1)$ will lead to the desired estimates),  then the Poincar\'{e}' inequality gives the desired estimate $(\ref{3.29})$.
\end{proof}

\begin{thm}Let $B_r=B(0,r)\subset \mathbb{R}^n$ be a ball with $r\in (20\varepsilon,1]$. Assume that $\mathcal{L}_\varepsilon$ satisfies the conditions $(\ref{1.2})$, $(\ref{1.3})$ and $(\ref{1.4})$. If $f\in L^2(B_r),$ and $g\in H^{3/2}(\partial B_r)$, let $u_\varepsilon,u_0\in H^1(B_r)$ be the weak solutions of $\mathcal{L}_\varepsilon u_\varepsilon=f$ and $\mathcal{L}_0 u_0=f$ in $B_r$, respectively,  with $u_\varepsilon=u_0=g$ on $\partial B_r$. Then there holds the following convergence rate estimates:
\begin{equation}\label{3.30}
||u_\varepsilon-u_0||_{L^2(B_r)}\leq C\left(\frac{\varepsilon}{r}\right)^{\frac{1}{2}}\{r||g||_{\dot{H}^{1/2}(\partial B_r)}+r^2||g||_{\dot{H}^{3/2}(\partial B_r)}+r^2||f||_{L^2(B_r)}\},
\end{equation} where $C$ depends on $\alpha,\beta,M$ and $n$.\end{thm}

\begin{proof}According to $(\ref{3.29})$, it gives  $$||u_\varepsilon-u_0||_{L^2(B_r)}\leq C r\{||\nabla u_0||_{L^2(B_r\setminus B_{r-5\varepsilon})}+\varepsilon||\nabla u_0||_{L^2(B_{r})}+\varepsilon||\nabla^2 u_0||_{L^2(B_{r})},$$
and the trace theorem  gives
\begin{equation}\label{3.31}\begin{aligned}
||\nabla u_0||^2_{L^2(B_r\setminus B_{r-5\varepsilon})}&=\int_0^{5\varepsilon}\int_{\partial B(0,r')}|\nabla u_0|^2dS_{r'}dr'\leq 5\varepsilon \sup_{r'\in[3r/4,r]}\int_{\partial B(0,r')}|\nabla u_0|^2dS_{r'}\\ &\leq C\varepsilon/r ||\nabla u_0||^2_{L^2(B_r)}+C\varepsilon r ||\nabla^2 u_0||^2_{L^2(B_r)}
\end{aligned}\end{equation}
Observing that $\varepsilon\leq(\frac{\varepsilon}{r})^{1/2}r\leq(\frac{\varepsilon}{r})^{1/2}$,  then
$$\begin{aligned}
||u_\varepsilon-u_0||_{L^2(B_r)}&\leq C r\left\{||\nabla u_0||_{L^2(B_r\setminus B_{r-5\varepsilon})}+\varepsilon||\nabla u_0||_{L^2(B_{r})}+\varepsilon||\nabla^2 u_0||_{L^2(B_{r})}\right\}\\
&\leq C r\left\{\left(\frac{\varepsilon}{r}\right)^{1/2}[ ||\nabla u_0||_{L^2(B_r)}+ r ||\nabla^2 u_0||_{L^2(B_r)}]+\varepsilon||\nabla^2 u_0||_{L^2(B_r)}\right\}\\
&\leq C r\left(\frac{\varepsilon}{r}\right)^{1/2}(||\nabla u_0||_{L^2(B_r)}+ r ||\nabla^2 u_0||_{L^2(B_r)})\\
&\leq C\left (\frac{\varepsilon}{r}\right)^{\frac{1}{2}}(r||g||_{\dot{H}^{1/2}(\partial B_r)}+r^2||g||_{\dot{H}^{3/2}(\partial B_r)}+r^2||f||_{L^2(B_r)}).
\end{aligned}$$
where we use the $H^2$ estimate of Laplace equation in the last inequality and thus complete the proof.
\end{proof}
Actually, we have obtained the following result.
\begin{thm} (convergence rates). Let $\Omega\subset \mathbb{R}^n$ be a bounded $C^{1,1}$ domain. Assume that $\mathcal{L}_\varepsilon$ satisfies the conditions $(\ref{1.2}),\ (\ref{1.3})$ and $(\ref{1.4})$. If $f\in L^2(\Omega),$ and $g\in H^{3/2}(\partial \Omega)$, let $u_\varepsilon,u_0\in H^1(\Omega)$ be the weak solutions of $(\ref{1.1})$ and $(\ref{1.5})$,  respectively, then there holds the following estimates \begin{equation*}
||u_\varepsilon-u_0||_{L^2(\Omega)}\leq C\varepsilon^{1/2}(||g||_{H^{3/2}(\partial \Omega)}+||f||_{L^2(\Omega)}),
\end{equation*} where $C$ depends on $\alpha,\beta,M,n$ and $\Omega$.\end{thm}

\section{Proof of Theorem 1.2}

In this section, we study the convergence rates in $L^2$ and give the proof of Theorem 1.2. (Actually, We follow the proof in \cite[Chapter 2.4]{shen2018periodic}).
We firstly consider the special case when $\Omega=B_r$ is a ball with radial $r\in (20\varepsilon,1]$, and want to identify the constants in what way depending on $r$. Let $A^*$ denote the adjoint of $A$; i.e., $A^*=(a_{ij}^*(y,z))=(a_{ji}(y,z))$. For $G\in L^2(B_r)$, let $v_\varepsilon$ be the weak solution to
\begin{equation}\label{4.1}
\left\{
\begin{aligned}
\mathcal{L}^*_{\varepsilon} v_{\varepsilon} \equiv-\frac{\partial}{\partial x_{i}}\left(a^*_{i j}\left(\frac{x}{\varepsilon}, \frac{x}{\varepsilon^{2}}\right) \frac{\partial v_\varepsilon}{\partial x_{j}}\right)&=G  \text { in } B_r \\
v_{\varepsilon}&=0  \text { on } \partial B_r,
\end{aligned}\right.
\end{equation}
and $\chi^{j,*}(y)$, $\chi_y^{j,*}(z)$ be the correctors, respectively. The homogenized equation is given by
\begin{equation}\label{4.2}
\left\{\begin{aligned}
\mathcal{L}_{0} v_0 \equiv-\operatorname{div}\left(\widehat{A^*}\nabla v_0\right) &=G  \text { in } B_r \\
v_{0}&=0  \text { on } \partial B_r,
\end{aligned}\right.
\end{equation} and $\widehat{A^*}$ satisfies the similar equality as $\widehat{A}$. Due to the $H^2$ estimate of $v_0$, we have
\begin{equation}\label{4.3}||\nabla v_0||_{L^2(B_r)}\leq Cr||G||_{L^2(B_r)},\end{equation}
and \begin{equation}\label{4.4}||\nabla^2 v_0||_{L^2(B_r)}\leq C||G||_{L^2(B_r)},\end{equation} where the constant $C$ depends on $\alpha,\beta,$ and $n$.
Let
\begin{equation}\label{4.5}\begin{aligned}
R_\varepsilon(x)=&v_\varepsilon(x)-v_0(x)+\varepsilon \chi^{j,*}(x/\varepsilon)\psi_{2\varepsilon}S_\varepsilon\left(\partial_{x_j}v_0\right)\\
&+\varepsilon^2\chi^{j,*}_y(x/\varepsilon^2)\left[\psi_{2\varepsilon}S_\varepsilon\left(\partial_{x_j}v_0\right)
-\partial_{y_j}\chi^{k,*}(x/\varepsilon)\psi_{2\varepsilon}S_\varepsilon\left(\partial_{x_k}v_0\right)\right],
\end{aligned}\end{equation} where $\chi^{k,*}_y(x/\varepsilon^2)=\chi^{k,*}(x/\varepsilon,x/\varepsilon^2)$.
By the trace theorem, $(\ref{4.3})$ and $(\ref{4.4})$, we have
\begin{equation}\label{4.6}\begin{aligned}
||\nabla v_0||_{L^2(B_r\setminus B_{r-5\varepsilon})}\leq& C\left(\frac{\varepsilon}{r}\right)^{1/2}(||\nabla v_0||_{L^2(B_r)}
+r||\nabla^2 v_0||_{L^2(B_r)})\\ \leq& C (\varepsilon r)^{1/2}||G||_{L^2(B_r)},
\end{aligned}\end{equation} where $C$ depends on $\alpha,\beta,M$ and $n$.
In view of Lemma 3.4 and noting that $\varepsilon\leq r<1$,we have the following estimate
\begin{equation}\label{4.7}\begin{aligned}
||\nabla R_\varepsilon||_{L^2(B_r)}\leq& C\left(\varepsilon||\nabla v_0||_{L^2(B_r)}+||\nabla
v_0||_{L^2(B_r\setminus B_{r-5\varepsilon})}+\varepsilon||\nabla^2 v_0||_{L^2(B_r)}\right)\\
\leq &C\left(\frac{\varepsilon}{r}\right)^{1/2}(||\nabla v_0||_{L^2(B_r)}+r||\nabla^2 v_0||_{L^2(B_r)})\\
\leq& C (\varepsilon r)^{1/2}||G||_{L^2(B_r)}.
\end{aligned}\end{equation}

Observe that \begin{equation}\label{4.8}
\left|\int_{B_r} w_{\varepsilon} \cdot G dx\right|=\left|\int_{B_r} a_{ij} \partial_j w_{\varepsilon}\partial_i v_{\varepsilon} dx\right|\leq J_{1}+J_{2}+J_{3}+J_4,
\end{equation} where
\begin{equation}\label{4.9}\begin{aligned}
J_1&+J_2=\left|\int_{B_r} a_{ij} \partial_j w_{\varepsilon}\partial_i R_{\varepsilon} dx\right|+\left|\int_{B_r} a_{ij} \partial_j w_{\varepsilon}\partial_i v_{0} dx\right| \\
J_3&=\left|\int_{B_r} a_{ij} \partial_j w_{\varepsilon}\partial_i\left\{\varepsilon
\chi^{j,*}(x/\varepsilon)\psi_{2\varepsilon}S_\varepsilon\left(\partial_{x_j}v_0\right)\right\}dx\right|=:\left|\int_{B_r} a_{ij} \partial_j w_{\varepsilon}\partial_iv_1dx\right| \\
J_4&=\left|\int_{B_r} a_{ij} \partial_j w_{\varepsilon}\partial_i\left\{\varepsilon^2\chi^{j,*}_y(x/\varepsilon^2)\left[\psi_{2\varepsilon}S_\varepsilon\left(\partial_{x_j}v_0\right)-\partial_{y_j}\chi^{k,*}(x/\varepsilon)\psi_{2\varepsilon}S_\varepsilon\left(\partial_{x_k}v_0\right)\right]
\right\}dx\right|\\
&=:\left|\int_{B_r} a_{ij} \partial_j w_{\varepsilon}\partial_iv_2dx\right|
 \end{aligned}
\end{equation}

According to $(\ref{3.28})$, $(\ref{3.31})$ and $(\ref{4.7})$,we have
\begin{equation}\label{4.10}
J_1\leq C ||\nabla w_\varepsilon||_{L^2(B_r)}||\nabla R_\varepsilon||_{L^2(B_r)}\leq C \varepsilon||G||_{L^2(B_r)}(||\nabla u_0||_{L^2(B_r)}+r||\nabla^2 u_0||_{L^2(B_r)}).
\end{equation}\\
Noting that $v_1,v_2\in H^1_0(\Omega)$ defined in $(\ref{4.9})$, then according to $(\ref{3.3})$, $(\ref{3.31})$, $(\ref{4.3})$, (\ref{4.4}) and (\ref{4.6}),  it gives
\begin{equation}\label{4.11}\begin{aligned}
J_2&\leq C\varepsilon(||\nabla u_0||_{L^2(B_r)}+||\nabla^2 u_0||_{L^2(B_r)})||\nabla v_0||_{L^2(B_r)}\\
&\quad \ +C||\nabla u_0||_{L^2(B_r\setminus B_{r-5\varepsilon})}||\nabla v_0||_{L^2(B_r\setminus B_{r-4\varepsilon})}\\
&\leq C\varepsilon (||\nabla u_0||_{L^2(B_r)}+r||\nabla^2 u_0||_{L^2(B_r)})||G||_{L^2(B_r)};
\end{aligned}\end{equation}\\
According to Lemma 2.7 (i), we can easily have
\begin{equation*}\begin{aligned}
||\nabla v_1||_{L^2(B_r)}\leq& ||\nabla_y\chi^{j,*}(x/\varepsilon)\psi_{2\varepsilon}S_\varepsilon\left(\partial_{x_j}v_0\right)||_{L^2(B_r)}+
\varepsilon||\chi^{j,*}(x/\varepsilon)\nabla(\psi_{2\varepsilon}S_\varepsilon\left(\partial_{x_j}v_0\right))||_{L^2(B_r)}\\
\leq& C||\nabla v_0||_{L^2(B_r)},\end{aligned}\end{equation*}
 and similarly, \begin{equation*}
||\nabla v_1||_{L^2(B_{r}\setminus B_{r-4\varepsilon})}\leq C||\nabla v_0||_{L^2(B_{r}\setminus B_{r-4\varepsilon})}\leq C (\varepsilon r)^{1/2}||G||_{L^2(B_r)},
\end{equation*}
then by Lemma 3.1, it gives
\begin{equation}\label{4.12}\begin{aligned}
J_3&\leq C\varepsilon(||\nabla u_0||_{L^2(B_r)}+||\nabla^2 u_0||_{L^2(B_r)})||\nabla v_1||_{L^2(B_r)}\\
&\quad \ +C||\nabla u_0||_{L^2(B_r\setminus B_{r-5\varepsilon})}||\nabla v_1||_{L^2(B_r\setminus B_{r-4\varepsilon})}\\
&\leq C\varepsilon (||\nabla u_0||_{L^2(B_r)}+r||\nabla^2 u_0||_{L^2(B_r)})||G||_{L^2(B_r)};
\end{aligned}\end{equation}
Similar for $J_4$, we have
\begin{equation}\label{4.13}J_4\leq C\varepsilon (||\nabla u_0||_{L^2(B_r)}+r||\nabla^2 u_0||_{L^2(B_r)})||G||_{L^2(B_r)}. \end{equation}
Consequently, combining $(\ref{4.10})-(\ref{4.13})$ leads to
\begin{equation}\label{4.14}\left|\int_{B_r} w_{\varepsilon} \cdot G dx\right|\leq C\varepsilon (||\nabla u_0||_{L^2(B_r)}+r||\nabla^2 u_0||_{L^2(B_r)})||G||_{L^2(B_r)},\end{equation}
then \begin{equation}\label{4.15}||w_\varepsilon||_{L^2(B_r)}\leq C\varepsilon (||\nabla u_0||_{L^2(B_r)}+r||\nabla^2 u_0||_{L^2(B_r)}),\end{equation}
In view of $(\ref{3.2})$, it finally gives that
\begin{equation}\label{4.16}||u_\varepsilon-u_0||_{L^2(B_r)}\leq C\varepsilon (||\nabla u_0||_{L^2(B_r)}+r||\nabla^2 u_0||_{L^2(B_r)}),\end{equation}
which completes the proof of Theorem 1.2 when $\Omega$ is a ball. Actually, the Theorem 1.2 still holds when $\Omega$ is a bounded $C^{1,1}$ domain.

\section{Interior Lipschitz estimates at large scale}
Firstly, we introduce the following approximate result, which shows that there exists a good function $w$ approximates $u_\varepsilon$ at a large scale.
\begin{lemma}
Let $\varepsilon<r<1/2.$ Assume that $\mathcal{L}_\varepsilon$ satisfies the conditions $(\ref{1.2})$, $(\ref{1.3})$ and $(\ref{1.4})$. And let $u_\varepsilon\in H^1(B(0,2r))$ be a weak solution of $\mathcal{L}_\varepsilon u_\varepsilon=f$ in $B(0,2r)$. Then there exists $w\in H^1(B(0,r))$ such that $$\mathcal{L}_0w=f,$$ and there holds
\begin{equation}\label{5.1}
\left(\fint_{B(0,r)}|u_\varepsilon-w|^2\right)^{1/2}\leq C\left(\frac{\varepsilon}{r}\right)^{1/4}\left\{\left(\fint_{B(0,2r)}|u_\varepsilon|^2\right)^{1/2}+r^2\left(\fint_{B(0,2r)}|f|^2\right)^{1/2}\right\},
\end{equation} where $C$ depends on $\alpha,\beta,M,$ and $n$.
\end{lemma}

\begin{proof} Let $B_r\triangleq B(0,r)$. Due to the Caccioppoli's inequality and the co-area formula, there exists $r_0\in[r,3r/2]$ such that
\begin{equation}\label{5.2}
\int_{\partial B(0,r_0)}|\nabla u_\varepsilon|^2dS\leq\frac{C}{r^3}\int_{B(0,2r)}|u_\varepsilon|^2dx+Cr\int_{B(0,2r)}|f|^2dx.
\end{equation}
Let $\delta\leq r$, consider the auxiliary equations, $\mathcal{L}_\varepsilon v_\varepsilon=f$ in $B(0,r_0)$ with $v_\varepsilon=(u_\varepsilon)_\delta$ on $\partial B(0,r_0)$; $\mathcal{L}_0 w=f$ in $B(0,r_0)$ with $w=(u_\varepsilon)_\delta$ on $\partial B(0,r_0)$; and
$\Delta z_\varepsilon=0$ in $B(0,r_0)$ with $z_\varepsilon=u_\varepsilon-(u_\varepsilon)_\delta$ on $\partial B(0,r_0)$, where $(u_\varepsilon)_\delta$ satisfies
\begin{equation}\label{5.3}(u_\varepsilon)_\delta\in H^{3/2}(\partial B_{r_0}) \text{ s.t.}\begin{cases}
||(u_\varepsilon)_\delta-u_\varepsilon||_{{L^2(\partial B_{r_0})}}\leq C \delta||u_\varepsilon||_{{\dot{H}^{1}}(\partial B_{r_0})}\\
||(u_\varepsilon)_\delta||_{{\dot{H}^{1/2}}(\partial B_{r_0})}\leq C||u_\varepsilon||_{{\dot{H}^{1/2}}(\partial B_{r_0})}\\ ||(u_\varepsilon)_\delta||_{{\dot{H}^{3/2}}(\partial B_{r_0})}\leq C \delta^{-1/2}||u_\varepsilon||_{{\dot{H}^{1}}(\partial B_{r_0})}\\
||(u_\varepsilon)_\delta-u_\varepsilon||_{{\dot{H}^{1/2}}(\partial B_{r_0})}\leq C \delta^{1/2}||u_\varepsilon||_{{\dot{H}^{1}}(\partial B_{r_0})}\end{cases}\end{equation}
Then \begin{equation}\label{5.4}\begin{aligned}
\int_{B(0,r)}|u_\varepsilon-w|^2dx&\leq \int_{B(0,r)}|u_\varepsilon-v_\varepsilon-z_\varepsilon|^2dx+\int_{B(0,r)}|v_\varepsilon-w|^2dx+\int_{B(0,r)}|z_\varepsilon|^2dx\\
&=:I_1^2+I_2^2+I_3^2.\end{aligned}\end{equation}
In fact, we have \begin{equation}\label{5.5}
\int_{B(0,r_0)}a_{ij}(x/\varepsilon,x/\varepsilon^2)(\partial_ju_\varepsilon-\partial_jv_\varepsilon)\partial_i \phi dx=0
\end{equation} for any $\phi\in H^1_0(B(0,r_0))$. Taking $\phi=u_\varepsilon-v_\varepsilon-z_\varepsilon$ and according to $(\ref{1.2})$, it gives
\begin{equation}\label{5.6}
||\nabla u_\varepsilon-\nabla v_\varepsilon||_{L^2(B(0,r_0))}\leq C||\nabla z_\varepsilon||_{L^2(B(0,r_0))},
\end{equation} where $C$ depends on $\alpha,\beta$. Then, combining Poincar\'{e}' inequality and $(\ref{5.6})$, it gives that
\begin{equation}\label{5.7}\begin{aligned}
I_1&\leq Cr||\nabla( u_\varepsilon-v_\varepsilon-z_\varepsilon)||_{L^2(B(0,r_0))}\leq Cr ||\nabla z_\varepsilon||_{L^2(B(0,r_0))}\\
&\leq Cr ||(u_\varepsilon)_\delta-u_\varepsilon||_{\dot{H}^{1/2}(\partial B_{r_0})}\leq C r\delta^{1/2}||u_\varepsilon||_{\dot{H}^{1}(\partial B_{r_0})}\\
&\leq Cr^{-1/2}\delta^{1/2}(||u_\varepsilon||_{L^2(B(0,r_0))}+r^2||f||_{L^2(B(0,r_0))}).
\end{aligned}\end{equation}Noting that $\delta\leq r<1/2$, and according to Theorem 3.5, it gives

\begin{equation}\label{5.8}\begin{aligned}
I_2&\leq ||v_\varepsilon-w||_{L^2(B(0,r_0))}\\
&\leq C \left(\frac{\varepsilon}{r}\right)^{\frac{1}{2}}\{r||(u_\varepsilon)_\delta||_{\dot{H}^{\frac{1}{2}}(\partial B_{r_0})}+r^2||(u_\varepsilon)_\delta||_{\dot{H}^{\frac{3}{2}}(\partial B_{r_0})}+r^2||f||_{L^2(B_{r_0})}\}\\
& \leq C \left(\frac{\varepsilon}{r}\right)^{\frac{1}{2}}\{r||u_\varepsilon||_{\dot{H}^{\frac{1}{2}}(\partial B_{r_0})}+r^2\delta^{-1/2}||u_\varepsilon||_{\dot{H}^1(\partial
B_{r_0})}+r^2||f||_{L^2(B_{r_0})}\}\\& \leq C
\left(\frac{\varepsilon}{r}\right)^{\frac{1}{2}}\{[||u_\varepsilon||_{L^2(B_{2r})}+r^2||f||_{L^2(B_{2r})}]\cdot(1+r^{1/2}\delta^{-1/2})+r||u||_{\dot{H}^1(B_{r_0})}\}\\
& \leq C\left(\frac{\varepsilon}{r}\right)^{\frac{1}{2}}(1+r^{1/2}\delta^{-1/2})\{||u_\varepsilon||_{L^2(B_{2r})}+r^2||f||_{L^2(B_{2r})}\}\\
&\leq C\left(\frac{\varepsilon}{\delta}\right)^{\frac{1}{2}}\{||u_\varepsilon||_{L^2(B_{2r})}+r^2||f||_{L^2(B_{2r})}\}
\end{aligned}\end{equation} where we use the trace theorem
$$r||u_\varepsilon||_{\dot{H}^{1/2}(\partial B_{r_0})}\leq C(n)||u_\varepsilon||_{L^2 (B_{r_0})}+C(n)r||u_\varepsilon||_{\dot{H}^1(B_{r_0})}$$
and $(\ref{5.2})$ in the fourth inequality, the Caccioppoli's inequality in the fifth inequality, and $\delta\leq r<1/2$ in the last inequality.
According to the properties of harmonic functions, it gives
\begin{equation}\label{5.9}\begin{aligned}
I_3& \leq Cr^{1/2}||z_\varepsilon||_{L^{\frac{2n}{n-1}}(B(0,r_0))}\leq Cr^{1/2}||(z_\varepsilon)^*||_{L^2(\partial B(0,r_0))}\\
& \leq Cr^{1/2}||z_\varepsilon||_{L^2(\partial B(0,r_0))}\leq Cr^{1/2}\delta||u_\varepsilon||_{\dot{H}^1(\partial B(0,r_0))}\\
& \leq Cr^{-1}\delta\{||u_\varepsilon||_{L^2(B_{2r})}+r^2||f||_{L^2(B_{2r})}\}.
\end{aligned}\end{equation}
where the notation $(z_\varepsilon)^*$ represents the nontangential maximal function of $z_\varepsilon$. Here the second inequality follows from \cite[Remark 9.3]{Kenig2010Homogenization}, and the third equation is the so-called nontangential maximal function estimate (see for example \cite[Theorem 8.5.14]{shen2018periodic}). We use the estimate $(\ref{5.3})$ in the fourth inequality and in the last step.\\
Thus, combining $(\ref{5.7})$, $(\ref{5.8})$ and $(\ref{5.9})$ gives that $$\begin{aligned}
||u_\varepsilon-w||_{L^2(B(0,r))}&\leq C\{(\varepsilon/\delta)^{1/2}+(\delta/r)^{1/2}+\delta/r\}\{||u_\varepsilon||_{L^2(B_{2r})}+r^2||f||_{L^2(B_{2r})}\}\\
& \leq C\{(\varepsilon/\delta)^{1/2}+(\delta/r)^{1/2}\}\{||u_\varepsilon||_{L^2(B_{2r})}+r^2||f||_{L^2(B_{2r})}\}\\
& \leq C(\varepsilon/r)^{1/4}\{||u_\varepsilon||_{L^2(B_{2r})}+r^2||f||_{L^2(B_{2r})}\}
\end{aligned}$$ where we set $\delta=(\varepsilon r)^{1/2}$, then $\delta \leq r \Leftrightarrow\varepsilon \leq r$. Therefore, we complete the proof.
\end{proof}

After we obtain the approximating lemma, the following story is totally similarly to \cite{shen2017boundary}. Before we proceed further, for any vector $M\in\mathbb{R}^n$, we denote $G(r,v)$ as the following
\begin{equation}\label{5.10}
G(r,v)\triangleq \frac{1}{r}\inf_{c\in \mathbb{R},M\in\mathbb{R}^n}\left\{ \left(\fint_{B(0,r)}|v-Mx-c|^2dx\right)^{1/2}+r^2\left(\fint_{B(0,r)}|f|^pdx\right)^{1/p}\right\}.
\end{equation}

\begin{lemma}Given $f\in L^p(\Omega)$ for some $p>n$, let $u_0\in H^1(B(0,2r))$ be a solution of $\mathcal{L}_0u_0=f$ in $B(0,2r)$. Then there exists a constant $\theta\in(0,1/4)$, depending only on ${\alpha'},$ $\beta$, $M$, $p$ and $n$ such that
\begin{equation}\label{5.11}G(\theta r,u_0)\leq \frac{1}{2}G(r,u_0) \text{\ \ \ \ holds for any }r\in(0,1).\end{equation}
\end{lemma}

\begin{proof}It is fine to assume $u_0\in H^2((B(0,r)))$. According to the De Giorgi-Nash-Moser theorem, there exists ${\alpha'}\in(0,1)$ and $C>1$, depending only on $\alpha,\beta,M$ and $n$, such that
\begin{equation}\label{5.12}
||\nabla u_{0}||_{C^{0,{\alpha'}}(B(0,r/2))}\leq C r^{-{\alpha'}}\left\{\frac{1}{r}\left(\fint_{B(0, r)}\left|u_{0}\right|^{2}\right)^{1 / 2}+r\left(\fint_{B(0, r)}|f|^{p}\right)^{1/p}\right\}.\end{equation}
then, according to the definition of $G(\theta r,u_0)$, we have
\begin{equation}\label{5.13}
\begin{aligned} G(\theta r, u_{0}) & \leq \frac{1}{\theta r}\left\{\left(\fint_{B(0, \theta r)}\left|u_{0}-M_0x-c_{0}\right|^{2}\right)^{\frac{1}{2}}+\theta^{2} r^{2}\left(\fint_{B(0, \theta r)}|f|^{p}\right)^{\frac{1}{p}}\right\}\\
&\leq C\theta^{\sigma}\left\{ r^{\alpha'}||\nabla u_{0}||_{C^{0,{\alpha'}}(B(0,r/4))}+r\left(\fint_{B(0,r)}|f|^p\right)^{1/p}\right\}\\
& \leq C \theta^{\sigma}\left\{\frac{1}{r}\left(\fint_{B(0, r)}\left|u_{0}\right|^{2}\right)^{1 / 2}+r\left(\fint_{B(0, r)}|f|^{p}\right)^{1/p}\right\},
\end{aligned}\end{equation}
where $\sigma=\min\{1-n/p,{\alpha'}\}$, and we choose $c_0=u_0(0)$ and $M_0=\nabla u_0(0)$ . Observing that $u_0-Mx-c$ satisfies the same equation as $u_0$ in $B(0,2r)$, therefore, we have
\begin{equation*}G(\theta r,u_0)\leq C\theta^\sigma G(r,u_0), \end{equation*}and we are done by choosing $\theta$ such that $C\theta^{\sigma}=1/2$.
\end{proof}

For simplicity, we denote $\Phi(r)$ by
\begin{equation}\label{5.14}
\Phi(r)=\frac{1}{r} \inf _{c \in \mathbb{R}}\left\{\left(\fint_{B_{(0,r)}}\left|u_{\varepsilon}-c\right|^{2}\right)^{1 / 2}+r^{2}\left(\fint_{B_{(0,r)}}|f|^{p}\right)^{1 / p}\right\}
\end{equation}
\begin{lemma}Assume the same conditions as in Theorem 1.3. Let $u_\varepsilon$ be a weak solution of $\mathcal{L}_\varepsilon(u_\varepsilon)=f$ in $B(0,2r)$. Then,  \begin{equation}\label{5.15}
G\left(\theta r, u_{\varepsilon}\right) \leq \frac{1}{2} G\left(r, u_{\varepsilon}\right)+C\left(\frac{\varepsilon}{r}\right)^{1 / 4}\Phi(2r)\end{equation} for any $\varepsilon\leq r<1/4$, where $\theta\in (0,1/4)$ is given in Lemma 5.2 and $C$ depends only on $\alpha,$ $\beta$, $M$, $p$ and $n$.
\end{lemma}
\begin{proof}
Fix $r\in[\varepsilon,1/4)$, let $w$ be a solution of $\mathcal{L}_0w=f$ in $B(0,r)$ as in Lemma 5.1. Then it gives
\begin{equation}\label{5.16}
\begin{aligned} G(\theta r, u_{\varepsilon}) & \leq \frac{1}{\theta r}\left(\fint_{B(0, \theta
r)}\left|u_{\varepsilon}-w\right|^{2}\right)^{\frac{1}{2}}+G(\theta r, w) \\
& \leq \frac{C}{r}\left(\fint_{B(0,r)}\left|u_{\varepsilon}-w\right|^{2}\right)^{\frac{1}{2}}+\frac{1}{2} G(r, w) \\
& \leq \frac{1}{2} G\left(r,u_{\varepsilon}\right)+\frac{C}{r}\left(\fint_{B(0, r)}\left|u_{\varepsilon}-w\right|^{2}\right)^{\frac{1}{2}} \\
& \leq \frac{1}{2} G\left(r,u_{\varepsilon}\right)+C\left(\frac{\varepsilon}{r}\right)^{1 / 4}\left\{\left(\frac{1}{r}\fint_{B(0,2 r)}\left|u_{\varepsilon}\right|^{2}\right)^{1 /2}+r\left(\fint_{B(0,2 r)}|f|^{p}\right)^{1 / p}\right\}\\
&\leq \frac{1}{2} G\left(r, u_{\varepsilon}\right)+C\left(\frac{\varepsilon}{r}\right)^{1 / 4}\Phi(2r), \end{aligned}
\end{equation} where we use the estimate $(\ref{5.11})$ in the second inequality, and $(\ref{5.1})$ in the fourth inequality, and note that $u_\varepsilon-c$ satisfy the same equation as $u_\varepsilon$ in $B(0,2r)$. Consequently, we complete the proof.
\end{proof}

At this position, we introduce the following iteration lemma which plays an important role in obtaining the Lipschitz estimates.
\begin{lemma}(Iteration lemma). Let $\psi(r)$ and $\Psi(r)$ be two nonnegative continuous functions on the intergral $(0,1]$. Let $0<\varepsilon<1/4$,
and suppose that there exists a constant $C_0$ such that
 \begin{equation}\label{5.17}
\begin{cases}{\max _{r \leq t \leq 2 r} \Psi(t) \leq C_{0} \Psi(2 r)} \\
\max _{r \leq s, t \leq 2 r}|\psi(t)-\psi(s)| \leq C_{0} \Psi(2 r),
\end{cases}
\end{equation} for any $r\in[\varepsilon,1/2)$.
We further assume that
\begin{equation}\label{5.18}
\Psi(\theta r) \leq \frac{1}{2} \Psi(r)+C_{0} w(\varepsilon / r)\{\Psi(2 r)+\psi(2 r)\}
\end{equation}
holds for any $\varepsilon\leq r<1/4$, where $\theta\in (0,1/4)$ is a constant and $w$ is a nonnegative increasing function on $[0,1]$ such that
$w(0)=0$ and
\begin{equation*}
\int_0^1\frac{w(t)}{t}dt<\infty.
\end{equation*}
Then, we have
\begin{equation}\label{5.19}
\max_{\varepsilon\leq r\leq 1}\left\{\Psi(r)+\psi(r)\right\}\leq C\left\{\Psi(1)+\psi(1)\right\},
\end{equation}
where $C$ depends only on $C_0$, $\theta$ and $w$.
\begin{proof}
The proof could be found in \cite[Lemma 6.4.6]{shen2018periodic}.
\end{proof}

\end{lemma}

\textbf{Proof of Theorem 1.3}. It is fine to assume $0<\varepsilon<1/4$, otherwise it follows from the classical elliptic theory. In view of $(\ref{5.10})$,  set $\Psi(r)=G(r,u_\varepsilon)$, $w(t)=t^{1/4}$. It is not hard to see that
\begin{equation}\label{5.20}
\Psi(t)\leq C\Psi(2r) \text{\quad if }t\in[r,2r].
\end{equation}
Next, define $\psi(r)=|M_r|$, where $M_r$
is the vector such that
\begin{equation}\label{5.21}
\Psi(r)=\frac{1}{r} \inf _{c \in \mathbb{R}}\left\{\left(\fint_{B(0, r)}\left|u_{\varepsilon}-M_{r} x-c\right|^{2}\right)^{\frac{1}{2}}+r^{2}\left(\fint_{B(0, r)}|f|^{p}\right)^{\frac{1}{p}}\right\}.
\end{equation}
Then we have \begin{equation}\label{5.22}
\Phi(r) \leq C\{\Psi(2 r)+\psi(2 r)\},
\end{equation} with $\Phi(r)$ defined in $(\ref{5.14})$.
This coupled with Lemma 5.3 leads to
\begin{equation}\label{5.23}
\Psi(\theta r) \leq \frac{1}{2} \Psi(r)+C_{0} (\varepsilon / r)^{1/4}\{\Psi(2 r)+\psi(2 r)\},
\end{equation}
for $\varepsilon\leq r<1/4$, which satisfies the condition $(\ref{5.18})$. To verify the condition $(\ref{5.17})$, let $t,s\in[r,2r]$, then
\begin{equation}\label{5.24}
\begin{aligned}\left|M_{t}-M_{s}\right| & \leq \frac{C}{r}\left(\fint_{B(0, r)}\left|\left(M_{t}-M_{s}\right) x-c\right|^{2}\right)^{\frac{1}{2}} \\ & \leq \frac{C}{t}\left(\fint_{B(0, r)}\left|u_{\varepsilon}-M_{t} x-c\right|^{2}\right)^{\frac{1}{2}}+\frac{C}{s}\left(\fint_{B(0, r)}\left|u_{\varepsilon}-M_{s} x-c\right|^{2}\right)^{\frac{1}{2}} \\ & \leq C\{\Psi(t)+\Psi(s)\} \leq C \Psi(2 r). \end{aligned}
\end{equation} Consequently, according to Lemma 5.4, for any $\varepsilon\leq r<1/4$, we have the following estimate
\begin{equation}\label{5.25}
\begin{aligned} \frac{1}{r} \inf _{c \in \mathbb{R}}\left(\fint_{B(0, r)}\left|u_{\varepsilon}-c\right|^{2}\right)^{\frac{1}{2}}+r \left(\fint_{B(0,
r)}\left|f\right|^{p}\right)^{\frac{1}{p}}& \leq\{\Psi(r)+\psi(r)\} \leq C\{\Psi(1)+\psi(1)\} \\ & \leq C\left\{\|\nabla
u_\varepsilon\|_{L^{2}(B(0,1))}+\|f\|_{L^{p}(B(0,1))}\right\}. \end{aligned}
\end{equation}
Therefore, the desired estimate $(\ref{1.10})$ follows from the Caccioppoli' inequality.

\begin{rmk}If we want to obtain the following estimate
\begin{equation*}
\left(\fint_{B\left(x_{0}, \varepsilon^2\right)}\left|\nabla u_{\varepsilon}\right|^{2}\right)^{1 / 2} \leq C\left\{\left(\fint_{B\left(x_{0}, 1\right)}\left|\nabla u_{\varepsilon}\right|^{2}\right)^{1 / 2}+\left(\fint_{B\left(x_{0}, 1\right)}|f|^{p}\right)^{1 / p}\right\},
\end{equation*} then the scale of the approximating lemma (Lemma 5.1) should decrease to $\varepsilon^2$, consequently, it seems that we should choose $\lambda=2$ in $(\ref{3.27})$. Unfortunately, careful computation shows that $\lambda=2$ can't provide us any results of convergence rates.
\end{rmk}
\normalem
\end{document}